\documentclass{amsart}

\usepackage{graphicx} 
\usepackage{import}
\usepackage[mark=***]{sectionbreak}

\usepackage{amsthm}
\usepackage{amsmath}

\usepackage{txfonts} 

\usepackage{./Packages/abreviations}
\usepackage[Color=no]{./Packages/environments}

\usepackage{hyperref}
\hypersetup{
	colorlinks=true,
	linkcolor=blue,
	citecolor=blue,
	filecolor=black,
	urlcolor=blue
}

\title{The variety of Lie algebra representations}
\author{Bruna Mariana Braido da Silva Percinotti}
\date{March, 2026}

\begin{document}

%%% Heading %%%{

    \begin{abstract}
We study the affine variety $L_{n}(\mfr{g})$ of Lie algebra representations, the collection of all homomorphisms from an arbitrary $n$-dimensional Lie algebra into a fixed real semi-simple Lie algebra $\mfr{g}$. Using techniques from real Geometric Invariant Theory, we equip this variety with a natural moment map and associated energy functional arising from the action of the real reductive group $GL(n,\mbb{R}) \times \iner(\mfr{g})$.

We analyze the critical points of the energy functional and describe their structure. In particular, we prove that every semi-simple pair, that is representations of semi-simple Lie algebras, will globally minimize the energy in its orbit. As consequences, we obtain an elementary proof of the rigidity of semi-simple homomorphisms and derive a new proof of the Mostow theorem on the existence of compatible Cartan involutions for semi-simple subalgebras. Subsequent results concerning the structure of critical points of higher energy are also obtained.
	\end{abstract}

\maketitle
\tableofcontents

%%% --------- %%%}
%%% Main Body %%%{

\section{Introduction}
    % !TeX spellcheck = en_US

In what follows we introduce techniques from the real version of Geometric Invariant Theory (GIT) in the study of representations of real semi-simple Lie algebras. More explicitly, we take a semi-simple Lie algebra $\mfr{g}$ and consider $L_{n}(\mfr{g})$, the affine variety of all pairs $(\mu,\phi)$, where $\mu$ is a Lie bracket on $\mbb{R}^{n}$ and $\phi$ is a Lie homomorphism $(\mbb{R}^{n},\mu) \to \mfr{g}$.

We shall follow closely a trend of GIT applications initiated in \cite{lauret2002momentmapvarietylie} to study Complex Lie algebras and latter continued in \cite{gorodski2023momentmapvarietyjordan} to complex Jordan algebras. In both those articles the authors took the Kirwan-Ness energy functional, which is constructed from a natural group action producing a "base change" in each algebra operation, and considered its critical points lying in either the variety of Lie brackets or Jordan brackets, respectively. In this light, we shall also consider a appropriate group action on $L_{n}(\mfr{g})$ and an associated energy functional. 

For a semi-simple Lie algebra $\mfr{g}$ we now describe $L_{n}(\mfr{g})$. First, take the following vector space:

	\begin{equation*}
V_{n}(\mfr{g}) = \Lin\left( \extp^{2} \mbb{R}^{n},\mbb{R}^{n}\right) \oplus \Lin(\mbb{R}^{n},\mfr{g})
	\end{equation*}

\noindent So a arbitrary element of $V_{n}(\mfr{g})$ will be a pair $(\mu,\phi)$ where $\mu$ is a skew-symmetric bilinear map on $\mbb{R}^{n}$ with values on $\mbb{R}^{n}$ and $\phi$ is a linear map $\mbb{R}^{n} \to \mfr{g}$. For those pairs, we may consider the following system of polynomial equations

	\begin{equation}
	\begin{cases}
\mu(X,\mu(Y,Z)) = \mu(\mu(X,Y),Z) + \mu(Y,\mu(X,Z)) \\
\phi\mu(X,Y) = [\phi(X),\phi(Y)]
	\end{cases}
\qquad \forall X,Y,Z \in \mbb{R}^{n} 
	\end{equation}

\noindent which will define a affine variety

	\begin{equation*}
L_{n}(\mfr{g}) \coloneqq \left\{(\mu,\phi) \in V_{n}(\mfr{g}) \suchthat (\mu,\phi) \text{ satisfies \lastequation{0}}  
	\right\}
	\end{equation*}
	
\noindent In other words, to each pair $(\mu,\phi) \in L_{n}(\mfr{g})$ we obtain a $n$-dimensional Lie algebra $\mfr{h}_{\mu} \coloneqq (\mbb{R}^{n},\mu)$ and a Lie homomorphism $\phi: \mfr{h}_{\mu} \to \mfr{g}$. 

Readily, one can observe that the polynomial equations in \lastequation{0} are homogeneous, thus we also have the option of working with the projective variety  

	\begin{equation*}
\mcal{L}_{n}(\mfr{g}) \coloneqq \frac{L_{n}(\mfr{g}) \smallsetminus \{(0,0)\}}{\mbb{R}^{\times}} \subseteq \mbb{P}(V_{n}(\mfr{g})).
	\end{equation*}

\noindent In the following we work simultaneously with $L_{n}(\mfr{g})$ and $\mcal{L}_{n}(\mfr{g})$, depending on the convenience.

Additionally, on $V_{n}(\mfr{g})$, we have the natural linear action of the real reductive group

	\begin{equation*}
GC_{n}(\mfr{g}) \coloneqq GL(n,\mbb{R}) \times \iner(\mfr{g})
	\end{equation*}

\noindent So if $(g,h) \in GC_{n}(\mfr{g})$ and $(\mu,\phi) \in V_{n}(\mfr{g})$ we will have an new element of $V_{n}(\mfr{g})$ 

	\begin{equation*}
(g,h) \cdot (\mu,\phi) = (g \cdot \mu, h\phi g^{-1})
	\end{equation*}

\noindent where on the first entree we have the natural action of "base change" on the space of skew bilinear forms and on second entree we perform composition. 

\subsubsection{Remark:} Whenever possible, and the context is sufficient for understanding, we choose to abbreviate $V_{n}(\mfr{g})$, $L_{n}(\mfr{g})$ and $GC_{n}(\mfr{g})$ to the simpler notation $V_{n}$, $L_{n}$ and $GC_{n}$, respectively. \\

Observe that the equations in \lastequation{0} are preserved by the action of this group. Thus both varieties $L_{n}$ and $\mcal{L}_{n}$ are left invariant by $GC_{n}$. From this action, and an appropriate choice of inner product in $V_{n}$, we shall see in \S 2.2 how to construct the Kirwan-Ness energy functional $E: \mcal{L}_{n} \to \mbb{R}$. 

Usually, said functional $E$ is the main tool of the Kirwan-Ness stratification theorem for the null-cone (see \cite{bohm2017realgeometricinvarianttheory}). Concerning our results, however, we will be more interested on the properties of $E$ and its critical points, henceforth called critical pairs. The justification, besides their importance in the stratification itself, will be the structure results we shall obtain for those. 

After setting up the GIT machinery we shall use in \S 2, we will start by analyzing the minimal points of $E$ and their properties in \S 3. In this context the main applications we shall derive from the GIT framework we constructed will be directed at the semi-simple pairs $(\mu,\phi) \in L_{n}(\mfr{g})$, i.e., with $\mfr{h}_{\mu}$ semi-simple. This is so because of the main result for this section:

	\begin{theorem}
Let $(\mu:\phi) \in \mcal{L}_{n}(\mfr{g})$ be a semi-simple pair. Then its $GC_{n}(\mfr{g})$-orbit contains a global minimal for the energy functional.
	\end{theorem} 

The idea main idea of the proof will be to first establish it for the especial case $\phi = 0$, which was done previously in \cite{lauret2002momentmapvarietylie}, and from this we will conclude that somewhere in set $\overline{GC_{n} \cdot (\mu,\phi)}$ the energy functional does attain its global minimal. And from this it will be a simple matter to use the Hilbert-Mumford criterion and a representation theory argument to prove that this actually happens inside the orbit.

After proving \theo{1.1} we shall start to see the consequences for a pair to be critical. In particular, we shall obtain some well know structural results for homomorphisms between semi-simple Lie algebras:

	\begin{theorem}
The $GC_{n}(\mfr{g})$-orbit of a semi-simple pair is open in $L_{n}(\mfr{g})$. 	
	\end{theorem}

The idea behind the proof will be to apply \theo{1.1} to conclude that on the closure of any orbit there can be at most one semi-simple pair. Then it will be a simple matter to exhibit the orbit of a semi-simple pair as the intersection of $L_{n}(\mfr{g})$ with a open set. 

The pairs with open orbits are also called rigid, for a more standard nomenclature. We note that \theo{1.2} is already know from the more general theorem proved in \cite{bams/1183527432} where rigidity is studied by means of a cohomology theory defined from a graded Lie algebra, see also \cite{bams/1183526023,bams/1183528646} for more specific details. Also, we note that the results in \cite{Richardson1967ART} do not directly apply, since the dimension of the image is not fixed. Our proof will have the advantage of being elementary and make no use of cohomology results.

The second consequence we shall obtain from \theo{1.1} is the Mostow’s theorem regarding the existence of compatible Cartan involutions for semi-simple Lie subalgebras of semi-simple Lie algebras: 

	\begin{theorem}{}\cite{MR69829}
Let $\mfr{g}$ be semi-simple Lie algebra and $\mfr{h} \leqslant \mfr{g}$ any semi-simple subalgebra. Then for any Cartan involution $\sigma':\mfr{h} \to \mfr{h}$ is the restriction of a Cartan involution $\sigma: \mfr{g} \to \mfr{g}$.
	\end{theorem}

The proof will again be elementary and will rely on \theo{1.1}. The idea is for a given semi-simple critical pair $(\mu,\phi)$ to show the existence of a pair of Cartan involutions $\sigma: \mfr{g} \to \mfr{g}$ and $\sigma':\mfr{h}_{\mu} \to \mfr{h}_{\mu}$ for which $\sigma\phi=\phi\sigma'$. Subsequently, we shall use the action of the group $GC_{n}$ to extend this to any pair in the same orbit and to show that $\sigma'$ can be any pre-selected Cartan involution.

We note that Mostow's theorem is logically equivalent to Karpelevich's theorem (see \cite{zbMATH03084006}) which asserts that a isometric action of semi-simple connected Lie group on a symmetric space of non-compact type has at least one totally geodesic orbit. This version of \theo{1.2} was most recently proved in \cite{Di_Scala_2009} by purely geometric techniques, in contrast to our purely algebraic proof.

In order to obtain the aforementioned applications, only the global minima for the energy functional and their properties were studied. Naturally, this brings about questions concerning the remaining critical pairs: Will they have any interesting properties? 

For those questions, we present a miscellanea of results in section \S4 where we explore more deeply the general structure of those high energy critical pairs. Those structural properties will be direct generalizations of the Lie algebra case done by J. Lauret in \cite{lauret2002momentmapvarietylie}.  

Indeed, we shall see that the study of the higher energy critical pairs is essentially reduced to the nilpotent case (see \theo{4.2} and \theo{4.5}). Additionally, we also prove the existence of compatible gradations for the critical pairs:

	\begin{theorem}		
Let $(\mu:\phi) \in \mcal{L}_{n}(\mfr{g})$ be a critical point of the energy functional. Then we may find Lie algebra gradations $\mfr{h}_{\mu} = \sum_{i \geqslant 0} \mfr{h}^{i}$ and $\mfr{g} = \sum_{i \in \mbb{Z}} \mfr{g}^{i}$ which are compatible with the homomorphism $\phi$ in the sense that $\phi[\mfr{h}^{i}] \leqslant \mfr{g}^{i}$ for $i$ any non-negative integer.
	\end{theorem}
	
\subsubsection*{Acknowledgements:} It was a pleasure to have Claudio Gorodski for the many fruitful discussions we had. Additionally, the author is thankful to Christoph B\"{o}hm for the helpful comments and suggestions regarding the writing and structure of this paper.

\setcounter{section}{1}
\section{The moment map for the variety of isometric actions}
    % !TeX spellcheck = en_US

%{

In this section we shall introduce the techniques from GIT we shall use on the following sections. To do so, fix an semi-simple Lie algebra $\mfr{g}$. Then the group

	\begin{equation*}
GC_{n}(\mfr{g}) \coloneqq GL(\mbb{R}^{n}) \times \iner(\mfr{g})
	\end{equation*}
	
\noindent is real reductive in the sense of \cite{bohm2017realgeometricinvarianttheory} and we may use the results therein to study the action of $GC_{n}$ on $L_{n}$. But first, in order to properly setup the GIT machinery, we shall make two remarks concerning the action of $GC_{n}$.
	
Firstly, there is the construction of the infinitesimal action of $GC_{n}$ done by its Lie algebra $\mfr{gc}_{n} = \mfr{gc}_{n}(\mfr{g})$. Indeed, we start by making use of the semi-simplicity of $\mfr{g}$ to obtain the following identification:
	
	\begin{equation*}
\mfr{gc}_{n}(\mfr{g}) \cong \mfr{gl}(\mbb{R}^{n}) \oplus \mfr{g},
	\end{equation*}
	
\noindent Then, by a simple process of differentiation, we see that an element $(A,v) \in \mfr{gc}_{n}$ acts on $V_{n}$ according to the rule
	
	\begin{equation*}
(A,v) \cdot (\mu,\phi) = (A\cdot\mu,\ad_{v}\phi-\phi A),
	\end{equation*}
	
\noindent where $A\cdot\mu$ is the bilinear $\mbb{R}^{n} \times \mbb{R}^{n} \to \mbb{R}^{n}$ given by
	
	\begin{equation*}
A\cdot\mu(X,Y) = \mu(AX,Y) + \mu(X,AY) - A\mu(X,Y).
	\end{equation*}
	
Secondly, there is the matter of the isotropy subgroups (and Lie subalgebras) of a pair in $L_{n}$. For those we make the very suggestive notation and let $\aut(\mu,\phi)$ be the isotropy subgroup of $(\mu,\phi) \in L_{n}(\mfr{g})$ (or its projectivization) and name its elements automorphisms of the pair $(\mu,\phi)$, or simply pair automorphisms. Similarly, we let $\der(\mu,\phi)$ be the Lie algebra of $\aut(\mu,\phi)$ and name its elements derivations of the pair $(\mu,\phi)$, or simply pair derivations.
	
The justification for such names relies on the observation that $\der(\mu,\phi)$ often behaves as the Lie algebra of derivations of a Lie algebra. For instance, similar to the case of Lie algebras, a pair $(\mu,\phi)$ will have some control over its isotropy. Indeed, we have a characterization of the pair derivations
	
	\begin{equation*}
(\delta,v) \in \der(\mu,\phi) \iff 
	\begin{cases}
\delta \in \der(\mu) \\
\ad_{v}\phi = \phi\delta
	\end{cases}.
	\end{equation*}
	
\noindent From which we can construct, as in the Lie algebra case, a Lie algebra homomorphism
	
	\begin{equation*}
\ad^{\mu,\phi}: \mfr{h}_{\mu} \to \der(\mu,\phi),
	\end{equation*}
	
\noindent which is given by $X \mapsto (\ad^{\mu}_{X},\phi(X))$. (Recall that $\mfr{h}_{\mu}$ denotes the Lie algebra $(\mbb{R}^{n},\mu)$). The image of $\ad^{\mu,\phi}$ shall be denoted by $\mfr{inn}(\mu,\phi)$ and its elements are the inner pair derivations. 
	
%}
\subsection{The moment map and the energy functional}

We now are ready to enter the territory of real GIT. As we mentioned in the introduction, the moment map and the energy functional are defined up to some choice of background inner product on $V_{n}$. And the right way to make such choice will be to fix an arbitrary inner product on $\mbb{R}^{n}$ and a Cartan involution $\theta$ on $\mfr{g}$. Then on $\mfr{g}$ the following inner product can be considered 

	\begin{equation*}
\langle X,Y \rangle \coloneqq - B_{\mfr{g}}(\theta X,Y),
	\end{equation*}

\noindent where $B_{\mfr{g}}$ is the Killing form of $\mfr{g}$. Note that from the Cartan involution $\theta$ its eigenspaces $\mfr{k}$ and $\mfr{p}$ for the eigenvalues 1 and -1, respectively, will form a Cartan decomposition for $\mfr{g}$. 

Then the inner product we constructed on $\mfr{g}$ together with the one in $\mbb{R}^{n}$ induces inner products on any other space constructed out of those vector spaces, in particular we have one in $V_{n}$ and one in the Lie algebra $\mfr{gc}_{n}$. 

The inner product we constructed on $\mfr{gc}_{n}$ is also induced from an $\ad$-invariant bilinear form $\Tilde{\beta}$ and a Cartan involution $\Tilde{\theta}$. Those are given by

	\begin{equation}
	\begin{cases}
\Tilde{\beta}((A,v),(B,w)) \coloneqq \Tr(AB) + B_{\mfr{g}}(v,w) \\
\Tilde{\theta}(A,v) \coloneqq  (-A^{*},\theta(v)).
	\end{cases}
	\end{equation}

\noindent where $A^{*}$ is the adjoint relative to the inner product of $\mbb{R}^{n}$. Then we obtain a Cartan decomposition $\mfr{gc}_{n}(\mfr{g}) = \Tilde{\mfr{k}} \oplus \Tilde{\mfr{p}}$ where

	\begin{equation*}
\Tilde{\mfr{k}} = \mfr{so}(n) \oplus \mfr{k} \qquad\qquad \Tilde{\mfr{p}} = Symm(n) \oplus \mfr{p}
	\end{equation*}

\noindent And this decomposition is compatible with the inner product on $V_{n}$ in the sense that the action of $(A,v) \in \mfr{gc}_{n}$ on $V_{n}(\mfr{g})$ will be (skew-) symmetrical if, and only if, $(A,v)$ is (in $\Tilde{\mfr{k}}$) in $\Tilde{\mfr{p}}$. Also, we note that on the Lie group level we have a global Cartan decomposition

	\begin{equation*}
GC_{n}(\mfr{g}) = OC_{n}(\mfr{g}) \exp[\mfr{p}],
	\end{equation*}

\noindent where $OC_{n}(\mfr{g})$ is the maximal compact subgroup $O(\mbb{R}^{n}) \times (O(\mfr{g}) \cap \iner(\mfr{g}))$. 

\subsubsection{Remark:} In order to not overload our notations, we shall make the convention that all those inner products constructed out of those in $\mfr{g}$ and $\mbb{R}^{n}$ should be denoted by $\langle -,- \rangle$ and have the context be sufficient to understand on which vector space we are. \\

And after this long preamble we can start introducing the GIT tools we have at our disposal. Those will be the moment map $m: \mcal{L}_{n} \to \Tilde{\mfr{p}}$ which is implicitly defined by the relation

	\begin{equation}
\langle m(\mu:\phi), (A,v) \rangle = \frac{\langle (A,v) (\mu,\phi), (\mu,\phi) \rangle}{\|(\mu,\phi)\|^{2}}.
	\end{equation}

\noindent And the energy functional $E: \mcal{L}_{n} \to \mbb{R}$ which is given by

	\begin{equation*}
E(\mu:\phi) \coloneqq \|m(\mu:\phi)\|^{2} 
	\end{equation*}

\noindent And although the functional $E$ is in general not a Morse-Bott function it is possible to use it to construct a stratification of $\mcal{L}_{n}$, we shall briefly do so in section 2.2. On this account, the critical points of $E$ are of particular interest. 

	\begin{definition}
A pair $(\mu:\phi)$ in $\mcal{L}_{n}(\mfr{g})$ (or $L_{n}(\mfr{g})$) is called critical if it is (it projectivization is) a critical point of $E$. When this is the case we say its orbit in $\mcal{L}_{n}(\mfr{g})$ (or $L_{n}(\mfr{g})$) is distinguished. Any pair contained in a distinguished orbit is also called distinguished.
	\end{definition}
	
\subsubsection{Remark:} The notion of a critical pair is completely dependent on the choice of inner product in $\mbb{R}^{n}$ and Cartan involution on $\mfr{g}$. However, the notion of a distinguished pair does not depend on those choices. The reason is that the group $GC_{n}$ acts transitively on the set of inner products of $\mbb{R}^{n}$ and Cartan involutions of $\mfr{g}$ and different choices for those will produce moment maps which are related under the adjoint representation $GC_{n} \to \mfr{gl}(\mfr{gc}_{n})$. \\ 

Concerning the distinguished orbits and the critical pairs we have the following characterization:

	\begin{proposition}
Let $(\mu:\phi) \in \mcal{L}_{n}(\mfr{g)}$. Then its orbit will be distinguished if, and only if, the restriction of $E$ to its orbit attains its minimal value at some point $(\nu:\psi)$. Furthermore, if $(\mu:\phi)$ is distinguished then the critical pairs in its orbit are composed of a single $OC_{n}(\mfr{g})$-orbit.
	
	\begin{proof}
See \textbf{Theorem 1.3} and \textbf{Corollary 9.4} of \cite{bohm2017realgeometricinvarianttheory}.
	\end{proof}
	\end{proposition}

To systematically search and study the critical pairs we need a more practical characterization of those. From Lemma 7.2 in \cite{bohm2017realgeometricinvarianttheory} we know what is the gradient of the energy:

	\begin{equation*}
\nabla E(\mu:\phi) = 4 \left( E(\mu:\phi)(\id,0) + m(\mu:\phi)\right) \cdot (\mu:\phi)
	\end{equation*}

\noindent Thus the critical pairs can be characterized as solutions to the following equation:

	\begin{equation*}
m(\mu:\phi) \cdot (\mu,\phi) = E(\mu:\phi)(\mu,\phi)
	\end{equation*}

As we shall learn, it is a lot easier to do calculations with the cone $L_{n}$ than with the projective variety $\mcal{L}_{n}$. For this reason we define a "quadratic" version for the moment map and the energy. Those will be the map $M: L_{n} \to \mfr{p}$ given implicitly by 

	\begin{equation*}
\langle M(\mu,\phi), (A,v) \rangle = \langle (A,v)\cdot (\mu,\phi), (\mu,\phi) \rangle
	\end{equation*}

\noindent and the quantity

	\begin{equation*}
k_{\mu,\phi} = \frac{\|M(\mu,\phi)\|^{2}}{\|(\mu,\phi)\|^{2}}
	\end{equation*}

\noindent Then both are quadratic in $(\mu,\phi)$ and are related to the moment map and energy by means of

	\begin{equation*}
m(\mu:\phi) = \frac{M(\mu,\phi)}{\|(\mu,\phi)\|^{2}} \qquad \text{ and } \qquad E(\mu:\phi) = \frac{k_{\mu,\phi}}{\|(\mu,\phi)\|^{2}}
	\end{equation*}

\noindent We also have a object that resembles the gradient of $E$:

	\begin{equation*}
(D,u)_{\mu,\phi} \coloneqq M(\mu,\phi) + k_{\mu,\phi}(\id,0).
	\end{equation*}

\noindent Indeed,

	\begin{equation*}
\nabla E(\mu:\phi) = \frac{4}{\|(\mu,\phi)\|^{2}} (D,u) \cdot 	(\mu:\phi)
	\end{equation*}

\noindent And we obtain another characterization which states that a pair $(\mu,\phi)$ will be critical if, and only if, 

	\begin{equation*}
(D,u)_{\mu,\phi} \in \der(\mu,\phi).
	\end{equation*}

\subsubsection{Remark:} In the future, we shall use the notation $D_{\mu,\phi}$ and $u_{\mu,\phi}$ in place of $(D,u)_{\mu,\phi}$ when necessary to talk about those objects in individuality. Also, we shall omit subscripts altogether whenever the context becomes sufficient for understanding. \\

Following those definitions we now present some basic properties of the moment map and the energy functional. In especial we shall obtain more explicit formulas for those.

	\begin{lemma}
Let $(\mu,\phi) \in L_{n}(\mfr{g})$ and $(\delta,v) \in \der(\mu,\phi)$. Then the following holds:
	
	\begin{enumerate}
		
		\item $\langle M(\mu,\phi), (\delta,v) \rangle = 0$.
		
		\item $\langle M(\mu,\phi), [(-\delta^{*},\theta v),(\delta,v)] \rangle \geqslant 0$ and equality holds if, and only if, $(-\delta^{*},\theta v)$ is a pair derivation for $(\mu,\phi)$.
		
	\end{enumerate}
	
	\begin{proof}
		
	\begin{enumerate}
			
		\item From the definition of $\der(\mu,\phi)$ we see that $(\delta,v)(\mu,\phi) = 0$ and, from the definition of $M(\mu,\phi)$ we see that
			
			\begin{equation*}
		\langle M(\mu,\phi), (\delta,v) \rangle = \langle (\delta,v)(\mu,\phi),(\mu,\phi) \rangle = 0.
			\end{equation*}
			
		\noindent As claimed.
			
		\item Again, using the definition of $M(\mu,\phi)$ and the previous item we see that
			
			\begin{align*}
		\langle M(\mu,\phi), [(-\delta^{*},\theta v),(\delta,v)] \rangle & = \langle [(-\delta^{*},\theta v),(\delta,v)](\mu,\phi),(\mu,\phi) \rangle \\
		& = \langle (-\delta^{*},\theta v)(\delta,v)(\mu,\phi),(\mu,\phi) \rangle \\
		& - \langle (\delta,v)(-\delta^{*},\theta v)(\mu,\phi),(\mu,\phi) \rangle \\
		& = \|(-\delta^{*},\theta v)(\mu,\phi)\|^{2}
			\end{align*}
			
		\noindent And both the claims follows.
			
	\end{enumerate}
	\end{proof}
	\end{lemma}

Also, we shall interested on the eigenvalues of the symmetric map $D_{\mu,\phi}$. Already on the case of a generic pair we can prove its trace is non-negative: 

	\begin{lemma}
For any $(\mu,\phi) \in L_{n}(\mfr{g})$ we have
	
	\begin{equation*}
k_{\mu,\phi} \Tr D_{\mu,\phi} = \|(D,u)_{\mu,\phi}\|^{2}.
	\end{equation*}
	
\noindent In particular $\Tr D_{\mu,\phi} \geqslant 0$ and equality holds if, and only if, $(D,u)_{\mu,\phi} = 0$.
	
	\begin{proof}
First, observe that
		
	\begin{align*}
\|(\mu,\phi)\|^{2} & = - \langle M(\mu,\phi),(\id,0) \rangle \\
& = \langle k_{\mu,\phi}(\id,0) - (D,u)_{\mu,\phi},(\id,0) \rangle \\
& = nk_{\mu,\phi} - \langle D_{\mu,\phi}, \id \rangle.
	\end{align*}
		
\noindent Hence
		
	\begin{equation}
nk_{\mu,\phi} = \|(\mu,\phi)\|^{2} + \langle D_{\mu,\phi}, \id \rangle
	\end{equation}
		
\noindent So if we calculate
		
	\begin{align*}
\|M(\mu,\phi)\|^{2} & = \langle (D,u)_{\mu,\phi} - k_{\mu,\phi}(\id,0), (D,u)_{\mu,\phi} - k_{\mu,\phi}(\id,0)\rangle \\
& = \|(D,u)_{\mu,\phi}\|^{2}-2k_{\mu,\phi}\langle D_{\mu,\phi}, \id \rangle + k^{2}_{(\mu,\phi)}\langle \id, \id \rangle \\
& = \|(D,u)_{\mu,\phi}\|^{2}-k_{\mu,\phi}\langle D_{\mu,\phi}, \id \rangle + k_{\mu,\phi}\|(\mu,\phi)\|^{2} \\
& = \|(D,u)_{\mu,\phi}\|^{2}-k_{\mu,\phi}\langle D_{\mu,\phi}, \id \rangle + \|M(\mu,\phi)\|^{2} \\
	\end{align*}
		
\noindent we conclude that
		
	\begin{equation*}
k_{\mu,\phi}\langle D_{\mu,\phi}, \id \rangle = \|(D,u)_{\mu,\phi}\|^{2}
	\end{equation*}
		
\noindent to obtain the result of the lemma.
	\end{proof}
	\end{lemma}

More can be said on the case the pair $(\mu,\phi)$ is critical, for not only does the symmetric map $D_{\mu,\phi}$ has non-negative trace but all of its eigenvalues are non-negative: 

	\begin{proposition}
Let $(\mu:\phi) \in \mcal{L}_{n}(\mfr{g})$ be critical. Then the derivation $D_{\mu,\phi}$ is positive semi-definite.
	
	\begin{proof}
The linear endomorphism $D = D_{\mu,\phi}$ is diagonalizable, since it is symmetric, and we can take $X \neq 0$ with $D(X) = cX$. Then 
		
	\begin{equation*}
c \ad_{X}^{\mu} = \ad_{D(X)}^{\mu} = [D,\ad^{\mu}_{X}],
	\end{equation*}
		
\noindent and
		
	\begin{equation*}
c\phi(X) = \phi D(X) = [u_{\mu,\phi},\phi(X)] 
	\end{equation*}
		
\noindent Thus
		
	\begin{equation*}
c\ad^{\mu,\phi}_{X} = \left[ (D,u)_{\mu,\phi},\ad^{\mu,\phi}_{X} \right]
	\end{equation*}
		
\noindent From \lemm{2.3} we see that
		
	\begin{align*}
c \left\| \ad^{\mu,\phi}_{X} \right\|^{2} & = \langle [D,\ad_{X}^{\mu}], \ad_{X}^{\mu} \rangle + \langle [u_{\mu,\phi},\phi(X)],\phi(X) \rangle \\
& = \langle D, [\ad^{\mu}_{X},\ad_{X}^{\mu*}] \rangle + \langle u_{\mu,\phi},[\theta\phi(X),\phi(X)]\rangle \\
& = \langle (D,u_{\mu,\phi}), ([\ad^{\mu}_{X},\ad_{X}^{\mu*}],[\theta\phi(X),\phi(X)]) \rangle \\
& = \langle M(\mu,\phi), [(\ad^{\mu}_{X},\phi(X)),(-\ad_{X}^{\mu*},\theta\phi(X))] \rangle \\
& \geqslant 0. 
	\end{align*}

\noindent Hence $c \geqslant 0$ provided $\ad_{X}^{\mu,\phi} \neq 0$. Otherwise, $\ad_{X}^{\mu} = 0$ and $\phi(X) = 0$ and we obtain that 
		
	\begin{equation*}
c \geqslant k_{\mu,\phi} > 0, 
	\end{equation*}
		
\noindent since 
		
	\begin{equation*}
0 \leqslant \langle M(\mu)(X),X \rangle = \langle D(X) -k_{\mu,\phi} X + \phi^{*}\phi(X),X \rangle = (c-k_{\mu,\phi}) \langle X,X \rangle
	\end{equation*}
		
\noindent (note the formula for $\langle M(\mu)(X),X \rangle$ in \prop{2.6}).
	\end{proof}
	\end{proposition}

This result which states that $D_{\mu,\phi}$ is positive semi-definite, when $(\mu,\phi)$ is critical, will latter be improved in \prop{4.3} when we prove that all its eigenvalues are non-negative rational numbers, up to a common multiplicative constant. 

Also, a explicit formula for the energy and moment map can be given.

	\begin{proposition}
Let $(\mu:\phi) \in \mcal{L}_{n}(\mfr{g})$. Then its energy and (quadratic) moment map are, respectively, 
	
	\begin{equation*}
E(\mu:\phi) = \frac{1}{n}\left( 1 + \frac{\Tr D_{\mu,\phi}}{\|(\mu,\phi)\|^{2}} \right) \qquad \text{and} \qquad M(\mu,\phi) = \left( M(\mu) - \phi^{*}\phi, \sum_{i} [\theta \phi(e_{i}),\phi(e_{i})]\right)
	\end{equation*}
	
\noindent Here $e_{1}, \dots, e_{n}$ is an orthogonal base for $\mbb{R}^{n}$ and $M(\mu)$ is the (quadratic) moment map for the Lie bracket $\mu$ given by the formula
	
	\begin{equation*}
\langle M(\mu)X,Y \rangle = \frac{1}{2}\sum_{i,j} \langle X,\mu(e_{i},e_{j}) \rangle\langle \mu(e_{i},e_{j}),Y \rangle - \sum_{i} \langle \mu(X,e_{i}),\mu(Y,e_{i}) \rangle
	\end{equation*}
	
	\begin{proof}
For the formula for $E$ note that from \lastequation{0} 
		
	\begin{equation*}
E(\mu:\phi) = \frac{k_{\mu,\phi}}{\|(\mu,\phi)\|^{2}} = \frac{1}{n}\left( 1 + \frac{\Tr D_{\mu,\phi}}{\|(\mu,\phi)\|^{2}} \right)
	\end{equation*}
		
And for the formula of $M(\mu,\phi)$ we note that
		
	\begin{align*}
\langle M(\mu,\phi), (A,v) \rangle & = \langle (A,v)\cdot (\mu,\phi), (\mu,\phi) \rangle \\
& = \langle A \cdot \mu,\mu \rangle + \langle (A,v)\phi,\phi \rangle \\
& = \langle M(\mu),A \rangle + \langle (A,v)\phi,\phi \rangle 
	\end{align*}
		
\noindent Then the calculation for the explicit formula of $M(\mu)$ can be found in the Proposition 3.2 from \cite{lauret2002momentmapvarietylie}, so working with the remaining term we obtain
		
	\begin{align*}
\langle (A,v)\phi,\phi \rangle & = \langle \ad_{v}\phi - \phi A, \phi \rangle \\
& = \sum_{i=1}^{n} \langle [v,\phi(e_{i})], \phi(e_{i}) \rangle - \langle A,\phi^{*}\phi \rangle \\
& = \sum_{i=1}^{n} \langle v, [\theta\phi(e_{i}),\phi(e_{i})] \rangle - \langle A,\phi^{*}\phi \rangle \\
& = \left\langle (A,v), \left(-\phi^{*}\phi,\sum_{i=1}^{n}[\theta\phi(e_{i}),\phi(e_{i})] \right)\right\rangle
	\end{align*}
		
\noindent The formula follows.
	\end{proof}
	\end{proposition}
	
\subsubsection{Remark:} Following this description of the moment map we are able to find an important class of critical pairs. Indeed, every pair obtained from the inclusion $\mfr{p} \hookrightarrow \mfr{g}$, where $\mfr{p}$ is the radical of a parabolic subalgebra of $\mfr{g}$, will be critical (assuming the inclusion preserves the inner product). This follows from the study of the moment map of such algebras done in \cite{tamaru2007parabolicsubgroupssemisimplelie}. Also, it follows from the results in \S 4 that both the parabolic subalgebra and its nilradical are critical.  

\subsection{Real Geometric Invariant Theory}

We end this section with a brief presentation of the main theorems of Real GIT that will be used. Since we are only interested in the action of $GC_{n}$ on the variety $\mcal{L}_{n}$ we shall tailor those results to this particular situation. Nevertheless, the interested reader should feel encouraged to read \cite{bohm2017realgeometricinvarianttheory} for a self-contained proof and discussion of real GIT for those with a background in differential geometry. Also, for the complex case, we mention the books \cite{mumford1994geometric,kirwanCohomologyQuotientsSymplectic1984} and L. Ness paper on the stratification \cite{nessStratificationNullCone1984}. 

For the time being, fix a real reductive subgroup $R$ of $GC_{n}$, whose Lie algebra is $\mfr{r}$, which is compatible with its Cartan decomposition in the sense that

	\begin{equation*}
R = \left( R \cap OC_{n} \right) \exp[\mfr{r}\cap \Tilde{\mfr{p}}]
	\end{equation*}

Concerning the action of $R$ on $L_{n}$ (and $\mcal{L}_{n}$) we make the following definition:

	\begin{definition}
Let $R \leqslant GC_{n}(\mfr{g})$ be a closed subgroup. Then we say a nonzero pair $\xi \in L_{n}(\mfr{g})$ is $R$-unstable, $R$-semi-stable, or $R$-poly-stable if $\overline{R \cdot \xi}$ contains 0, does not contain 0, or equals the orbit $R \cdot \xi$, respectively. A similar definition is made for points in $\mcal{L}_{n}(\mfr{g})$ by first lifting it to $L_{n}(\mfr{g})$ and then observing the stability of the resulting point. 
	\end{definition}

As it is the case of $GC_{n}$, the action of the group $R$ on $\mcal{L}_{n}$ will produce its own moment map $m_{R}: \mcal{L}_{n} \to \mfr{r} \cap \Tilde{\mfr{p}}$ which is given by the same implicit formula used before in \eq{2.1}. It is not hard to see that $m_{R}$ is also given by the composition of $m$ with the orthogonal projection $\mfr{gc}_{n} \to \mfr{r}$. The Kempf-Ness theorem characterizes the $R$-poly-stable points in $L_{n}$ in terms of the moment map $m_{R}$.

	\begin{theorem}[Kempf-Ness]
A point $\xi \in L_{n}(\mfr{g})$. Then the following are equivalent:
	
	\begin{enumerate}
		
		\item The point $\xi$ is $R$-poly-stable.
		
		\item On the orbit $R\cdot\xi$ there is a point $\xi'$ with $m_{R}(\xi') = 0$.
		
		\item On the orbit $R\cdot\xi$ there is a point $\xi'$ with $\|\xi'\| \leqslant \|g\xi\|$ for every $g \in R$.
		
		\item On the closure of any $R$-orbit it contains an unique closed $R$-orbit. 
		
	\end{enumerate}
	
	\begin{proof}
See \textbf{Theorem 1.1} of \cite{bohm2017realgeometricinvarianttheory}.
	\end{proof}
	\end{theorem}

The second theorem from real GIT we present is the Hilbert-Mumford Criterion. In practice, it reduces the study of degenerations and stability to the problem of understanding the action of the one parameter subgroups of $R$. 

	\begin{theorem}[Hilbert-Mumford Criterion]
For each point $\xi \in L_{n}(\mfr{g})$ and each $\xi' \in \overline{R \cdot \xi}$ there is a $g \in R$ and an $\alpha \in \mfr{r} \cap \Tilde{p}$ such that
	
	\begin{equation*}
\lim_{t\to\infty} e^{t\alpha} \xi = g\xi'
	\end{equation*}
	
	\begin{proof}
See \textbf{Lemma 6.3} of \cite{bohm2017realgeometricinvarianttheory}.
	\end{proof}
	\end{theorem}

From those theorems we can extract a small result which will used in the next section: 

	\begin{lemma}
Suppose $\xi \in L_{n}(\mfr{g})$ is $R$-poly-stable point and let $\alpha \in \mfr{r} \cap \mfr{p}$ and $g \in R$ be such that
	
	\begin{equation*} 
\lim_{t\to\infty}e^{t\alpha} \xi = g \xi.
	\end{equation*}
	
\noindent Then we are able to conclude that $\alpha \in \der(\xi)$ and $g \in \aut(\xi)$.
	
	\begin{proof}
By the Kempf-Ness theorem there is a $h \in R$ for which $h\xi$ has minimal norm in $R\xi$. By defining $\alpha' = \Ad_{h}(\alpha)$ we see 
		
	\begin{equation*}
\lim_{t\to\infty}e^{t\alpha'} (h\xi) = hg \xi.
	\end{equation*}
		
\noindent The convexity of the function $t \mapsto e^{t\alpha'} (h\xi)$, see Lemma 5.1 and Corollary 5.2 of \cite{bohm2017realgeometricinvarianttheory}, implies that $\|h\xi\| \geqslant \|hg \xi\|$. But since $h\xi$ has minimal norm in its $R$-orbit we are lead to the conclusion that $\|h\xi\| = \|hg \xi\|$. Again by the same Corollary 5.2 we obtain that $h\xi = hg \xi$. Thus $g \in \aut(\xi)$ and for every real $s$ 
		
	\begin{equation*}
e^{s\alpha}\xi = \lim_{t\to\infty} e^{(t+s)\alpha}\xi = \xi
	\end{equation*}
		
\noindent so $\alpha \in \der(\xi)$.
	\end{proof}   
	\end{lemma}

Finally, the last result we mention is the Kirwan-Ness Stratification Theorem. The main idea behind this theorem is to identify, for each of its points, a direction $\beta \in \mfr{p}$ "most responsible" for its instability, and then allocate all the points with the same direction of "most instability" into what will be the strata $\mcal{S}_{\beta}$. To do so, we describe the points in $\mcal{S}_{\beta}$ as those which are semi-stable under the group $R_{\beta}$ which is obtained by "removing" the unstable direction $\beta$.

Fix $\beta \in \mfr{p}$ and denote by $\beta^{+}$ the vector $\beta + \|\beta\|^{2}(\id,0)$. Then we consider following two subsets of our projective variety:

	\begin{equation*}
V_{\beta}^{+} = \left\{ \xi \in \mcal{L}_{n} \suchthat \langle m(\xi), \beta^{+} \rangle \geqslant 0 \right\} \qquad \text{and} \qquad V_{\beta}^{0} = \left\{\xi \in \mcal{L}_{n} \suchthat \langle m(\xi), \beta^{+} \rangle = 0 \right\}
	\end{equation*}

\noindent Then one can prove that for each non-zero point $\Tilde{\xi} \in \xi$, with $\xi \in V_{\beta}^{+}$, the limit 

	\begin{equation*}
\lim_{t\to\infty} \exp\left(-t\beta^{+}\right) \Tilde{\xi}
	\end{equation*}

\noindent will converge to a non-zero element of $L_{n}$ whose generated line in $\mcal{L}_{n}$, call it $p_{\beta}(\xi)$, will be a point in $V_{\beta}^{0}$. So this will produce a map

	\begin{equation*}
p_{\beta}: V_{\beta}^{+} \to V_{\beta}^{0}
	\end{equation*}

On the other hand, let $Z_{\beta}$ be the centralizer in $GC_{n}$ of $\beta$, and $\mfr{z}_{\beta}$ its Lie algebra. Then we define the following (reductive) subgroup: 

	\begin{equation*}
R_{\beta} \coloneqq \left( Z_{\beta} \cap OC_{n}\right)\exp\left[ \mfr{p} \cap \mfr{z}_{\beta} \cap \beta^{\perp} \right]
	\end{equation*} 

\noindent And the set of $R_{\beta}$-semi-stable points 

	\begin{equation*}
U_{\beta}^{0} = \left\{ \xi \in V_{\beta}^{0} \suchthat \xi \text{ is $R_{\beta}$-semi-stable } \right\}
	\end{equation*}

Finally, associated to $\beta$ we will consider the following subset of $\mcal{L}_{n}(\mfr{g})$: 

	\begin{equation*}
\mcal{S}_{\beta} \coloneqq OC_{n} \cdot p_{\beta}^{-1} \left[ U_{\beta}^{0} \right] 
	\end{equation*}

Save for few many cases, most choices of $\beta$ will produce an empty $\mcal{S}_{\beta}$, those that do not are characterized by the energy functional. Denote by $\mcal{C} \subset \mcal{L}_{n}$ the set of critical points of the functional $E$. Then one can prove that $m[\mcal{C}]$ is the union of finitely many $OC_{n}$-orbits and that $\mcal{S}_{\beta} \neq \varnothing$ if, and only if, $\beta \in m[\mcal{C}]$. Even more, it can be proved that two given vectors $\beta,\beta' \in m[\mcal{C}]$ are in the same $OC_{n}$-orbit if, and only if, both have the same norm and this is the case if, and only if, $\mcal{S}_{\beta} = \mcal{S}_{\beta'}$. For this reason we are able to use the simpler notation $\mcal{S}_{\|\beta\|^{2}}$ in place of $\mcal{S}_{\beta}$. The following theorem further classifies the relationship between the stratification and the energy functional:

	\begin{theorem}[Kirwan-Ness Stratification]
The set $\mcal{B}$ of critical values of the energy functional $E: \mcal{L}_{n}(\mfr{g}) \to \mbb{R}$ is finite. Furthermore, the collection of sets $\{\mcal{S}_{c} \suchthat c \in \mcal{B}\}$ possess the following properties: 
	
	\begin{enumerate}
		
		\item $\mcal{L}_{n}(\mfr{g}) = \bigsqcup_{c \in \mcal{B}} \mcal{S}_{c}$ (disjoint union).
		
		\item $\overline{\mcal{S}_{c}} = \bigsqcup_{c' \geqslant c} \mcal{S}_{c'}$
		
		\item Each strata $\mcal{S}_{c}$ is a $GC_{n}(\mfr{g})$-invariant and a intersection of $\mcal{L}_{n}(\mfr{g})$ with a smooth manifold.
		
		\item A point $\xi \in \mcal{L}_{n}(\mfr{g})$ is in $\mcal{S}_{c}$ if, and only if, the integral line of $-\nabla E$ starting at $\xi$ converges to a critical point $\xi_{\infty} \in \mcal{C}$ with $E(\xi_{\infty}) = c$.
		
	\end{enumerate}
	
	\begin{proof}
See \textbf{Theorem 1.2} of \cite{bohm2017realgeometricinvarianttheory}.
	\end{proof}
	\end{theorem}

\setcounter{section}{2}
\section{The stratum of lowest energy}
     % !TeX spellcheck = en_US

%{
	
The energy $E$ is a continuous function defined on the compact space $\mcal{L}_{n}$. Thus we know that $E$ attains its global minimum in at least one pair. In this section we shall study those critical pairs and the associated strata in great detail to obtain a (almost) complete understand of them. As a application we shall obtain the rigidity of semi-simple subalgebras (\theo{1.2}) and Mostow's Theorem (\theo{1.3}).

Recall from \prop{2.6} the formula for the energy of a generic pair in $\mcal{L}_{n}$:

    \begin{equation}
E(\mu:\phi) = \frac{1}{n} \left( 1 + \frac{\Tr D_{\mu,\phi}}{\|(\mu,\phi)\|^{2}}\right)
    \end{equation}

\noindent On the other hand, observe that in \lemm{2.4} we proved that $\Tr D_{\mu,\phi} \geqslant 0$. This shows that the energy functional on $\mcal{L}_{n}$ is always bounded below by $\tfrac{1}{n}$ and one may wonder when this value is actually attained. From the Stratification Theorem, one can see that this will be the case if, and only if, the action of the subgroup 

    \begin{equation*}
SC_{n}(\mfr{g}) \coloneqq SL(n,\mbb{R}) \times \iner(\mfr{g})
    \end{equation*}

\noindent has any closed orbits. As we shall learn, this will always be the case save for a few exceptional choices of $\mfr{g}$ and $n$.   

    \begin{lemma}
For $(\mu:\phi) \in \mcal{L}_{n}(\mfr{g})$ the following are equivalent:

    \begin{enumerate}

        \item $D_{\mu,\phi} = 0$.
    
        \item $E(\mu:\phi) = \tfrac{1}{n}$.

        \item $E(\mu:\phi)(\id,0) + m(\mu:\phi) = 0$.

        \item $SC_{n} \cdot (\mu,\phi)$ is closed.
        
    \end{enumerate}

\noindent Furthermore, if one of those conditions holds true then $(\mu,\phi)$ will be a global minimum of the energy on $\mcal{L}_{n}(\mfr{g})$ and its algebra of pair derivations will be $\Tilde{\theta}$-invariant:

    \begin{equation}
(\delta,v) \in \der(\mu,\phi) \iff \Tilde{\theta}(\delta,v) = (-\delta^{*},\theta(v)) \in \der(\mu,\phi)
    \end{equation}

    \begin{proof}
Indeed, (1) and (2) are equivalent because of the formula \eq{3.1} and \lemm{2.4}. Also, by the same lemma and definition of $(D,u)_{\mu,\phi}$, we see that (1) and (3) are equivalent. The equivalence of (3) and (4) follows from the Stratification Theorem. Now, if any of the conditions (1)-(4) holds true then \lastequation{0} follows from (3) and \lemm{2.3}.
    \end{proof}
    \end{lemma}

Hence at the lowest energy level the moment map of a critical pair takes a very nice form, that being $-\tfrac{1}{n}(\id,0)$. Because of this fact the strata $\mcal{S}_{\frac{1}{n}}$ will be strongly limited in terms of the degenerations which occur between its points. Indeed, one can already note a compatibility between the algebraic and metric structure of a critical pair by observing \lastequation{0}. As we shall see, this is the cause of very strong constrains in both structures of the pairs in $\mcal{S}_{\frac{1}{n}}$ and it shall be the cornerstone of the results that will follow. 

We start with abelian pairs $(0:\phi)$ of $\mcal{S}_{\frac{1}{n}}$. Recall the definition from chapter 8 of \cite{humphreysIntroductionLieAlgebras1973} that a Lie subalgebra consisted of only semi-simple elements is called a toral subalgebra. Also, since we are working with the field of real numbers, an abelian subalgebra of $\mfr{g}$ is toral if, and only if, it is contained in some Cartan subalgebra. 

    \begin{theorem}
Let $(0:\phi) \in \mcal{L}_{n}(\mfr{g})$. Then $E(0:\phi) = \tfrac{1}{n}$ precisely when $\phi$ is a homothety and $\ima\phi$ is contained in some $\theta$-invariant Cartan subalgebra. Consequently, the orbit of the pair $(0:\phi)$ will contain a critical pair with energy $\tfrac{1}{n}$ if, and only if, $\ker\phi = 0$ and $\ima\phi$ is toral. 

    \begin{proof}

First suppose that $\phi$ is injective and its image is a toral subalgebra of $\mfr{g}$. We may find a $h \in \iner \mfr{g}$ and a $\theta$-invariant Cartan subalgebra $\mfr{t}$ of $\mfr{g}$ such that $h[\ima\phi] \leqslant \mfr{t}$ and then define $g = \sqrt{\phi^{*}h^{*}h\phi}$. So we conclude that $\psi = h \phi g^{-1}$ is a homothety and its image is contained in the $\theta$-invariant subalgebra $\mfr{t}$. Note then that $[\theta\psi(X),\psi(X)] = 0$ for all $X$, which implies that $u_{0,\psi} = 0$ and

    \begin{equation*}
m(0:\psi) = \frac{1}{\|(0,\psi)\|^{2}}(M(0)-\psi^{*}\psi,0) = -\frac{1}{n}(\id,0).
    \end{equation*}

\noindent Hence $E(0:\psi) = \tfrac{1}{n}$ and we conclude that the orbit of the pair $(0:\phi)$ is distinguished and it is contained in the strata $\mcal{S}_{\frac{1}{n}}$.

Conversely, suppose that $E(0:\phi) = \tfrac{1}{n}$. So we note that $D_{0,\phi} = 0$ which implies that

    \begin{align*}
\phi^{*}\phi & = D_{0,\phi} + \phi^{*}\phi \\
& = k_{0,\phi}\id  + M(0) \\
& = k_{0,\phi}\id.
    \end{align*}

\noindent And we see that $\phi$ is a homothety. Additionally, one sees that $(0,\theta\phi(X)) \in \der(0,\phi)$ for every $X$, by virtue of \lastequation{0}. Consequently, 

    \begin{equation*}
[\theta\phi(X), \phi(Y)] = 0, \qquad X,Y \in \mfr{h}_{0}
    \end{equation*}

\noindent and so $\ima\theta\phi$ centralizes $\ima\phi$. It is also easy to see that $\ima\phi$ is contained in some $\theta$-invariant Cartan subalgebra since the vector space spawned by $\{\theta\phi(X),\phi (X) \mid X \in \mfr{h}_{0}\}$ is a $\theta$-invariant abelian subalgebra of $\mfr{g}$. Consequently, any other pair $(0,\psi)$ in the same orbit as $(0,\phi)$ will be such that $\ima\psi$ is toral.
    \end{proof}
    \end{theorem}

Then we have found a class of distinguished pairs in $\mcal{S}_{\frac{1}{n}}$ which are noted to posses an intimate relation with the algebraic and metric structure of $\mfr{g}$. To find the remaining ones we will need the technical result, which will also be used in the latter section. 

    \begin{lemma}
For $(\mu:\phi) \in \mcal{L}_{n}(\mfr{g})$ let $\mfr{r}$ be a subalgebra of $\mfr{h}_{\mu}$ for which

    \begin{equation*}
\Tilde{\theta}\ad^{\mu,\phi}_{X} = (-\ad_{X}^{\mu*}, \theta\phi(X)) \in \der(\mu,\phi)
    \end{equation*}

\noindent whenever $X \in \mfr{r}$ and for which its orthogonal complement $\mfr{r}^{\perp}$ is $\ad^{\mu}[\mfr{r}]$-invariant. Then 

    \begin{enumerate}

        \item $\mfr{r}$ is a reductive Lie algebra.
        
        \item For every ideal $\mfr{i}$ of $\mfr{r}$ its orthogonal complement in $\mfr{r}$ is also a ideal.
    
        \item The endomorphism $\phi^{*}\phi$ is $\ad^{\mu}[\mfr{r}]$-invariant.

    \end{enumerate}

    \begin{proof}
\hfill\par
    \begin{enumerate}

        \item For $X \in \mfr{r}$ we have, by hypothesis, that $\ad_{X}^{\mu*}$ is a derivation of $\mfr{r}$ and so it must leave the subalgebra $\mfr{s} \coloneqq [\mfr{r},\mfr{r}]$ invariant. In turn, this implies that $\ad^{\mu}_{X}$ must leave its orthogonal complement $\mfr{a} = \mfr{s}^{\perp} \cap \mfr{r}$ invariant as well. We conclude that $\mfr{a}$ is a abelian ideal of $\mfr{r}$. Naturally, the same argument applies to $\mfr{s}$ proving this must be a perfect Lie algebra.

        Now, let $B$ be the killing form of $\mfr{s}$. The subspace $\ker B$ is a ideal of $\mfr{s}$ which is left invariant by any derivation of $\mfr{s}$. Consequently, $\ker B^{\perp}$ must be invariant by any derivation of $\mfr{s}$ as well, in particular $\ker B^{\perp}$ is a ideal of $\mfr{s}$. It now follows that $\ker B$ is a solvable Lie algebra, since its killing form vanishes, and that it is also perfect. Therefore, $\ker B = 0$ and $\mfr{s}$ is a semi-simple Lie algebra.

        \item Take $X \in \mfr{i}^{\perp} \cap \mfr{r}$. Then to each $Y \in \mfr{r}$ we may find a $Y'$ also in $\mfr{r}$ for which

            \begin{equation*}
        \ad_{Y}^{*}|_{\mfr{r}} = \ad_{Y'}|_{\mfr{r}}. 
            \end{equation*}

        \noindent This follows from (1) and the hypothesis. Then to each $Z \in \mfr{i}$ we have

            \begin{equation*}
        \langle [Y,X],Z \rangle = \langle X, [Y',Z] \rangle = 0.
            \end{equation*}

        \noindent So $[Y,X] \in \mfr{i}^{\perp} \cap \mfr{r}$.

        \item This is a simple calculation: Fix $X \in \mfr{r}$ and let $Y,Z \in \mfr{h}_{\mu}$ be arbitrary. We see that

            \begin{align*}
        \langle \phi^{*}\phi \ad^{\mu}_{X} Y,Z \rangle & = \langle \ad_{\phi X} \phi Y, \phi Z \rangle \\
        & = -\langle \phi Y, \ad_{\theta\phi X}\phi Z \rangle \\
        & = \langle \phi Y, \phi\ad_{X}^{\mu*} Z \rangle \\
        & = \langle \ad^{\mu}_{X}\phi^{*}\phi Y,Z \rangle
            \end{align*}

        \noindent Thus for every $X \in \mfr{r}$ it holds that

            \begin{equation*}
        [\phi^{*}\phi,\ad_{X}] = 0.
            \end{equation*}
    \end{enumerate}    
    \end{proof}
    \end{lemma}

In our context of working with critical pairs $(\mu,\phi)$ in $\mcal{S}_{\frac{1}{n}}$ this last lemma can be applied to $\mfr{r} = \mfr{h}_{\mu}$ to show that all of them must come from reductive Lie algebras and that some decomposition could be done. Indeed, we shall have the following:

    \begin{proposition}
Let $(\mu:\phi) \in \mcal{S}_{\frac{1}{n}}$ be a critical pair. Then $\mfr{h}_{\mu}$ is a reductive Lie algebra which can be orthogonally decomposed as a sum of its center $\mfr{h}_{0}$ with its simple ideals $\mfr{h}_{i}$, $i=1, \dots, k$. Such decomposition makes each restriction $\phi_{i} \coloneqq \phi|_{\mfr{h}_{i}}$ into a homothety such that

    \begin{equation}
\ima\phi_{i} \perp \ima\phi_{j}. 
    \end{equation}

\noindent whenever $i \neq j$.

Furthermore, the pair $(\mu_{i}:\phi_{i}) \in \mcal{L}_{n_{i}}(\mfr{g})$ corresponding to the Lie algebra $\mfr{h}_{i}$ and the restriction $\phi_{i}$ is a critical pair with energy $\frac{1}{n_{i}}$.

    \begin{proof}
Since $(\mu:\phi)$ is a critical pair for the strata $\mcal{S}_{\frac{1}{n}}$ the derivation $D_{\mu,\phi}$ must be zero and the hypotheses of \lemm{3.3} holds for $\mfr{r} = \mfr{h}_{\mu}$. It follows that $\mfr{h}_{\mu}$ is a reductive Lie algebra which is the orthogonal sum of its center $\mfr{h}_{0}$ with its simple ideals $\mfr{h}_{i}$, $i=1, \dots, k$. Also, from the same lemma, we see that $\phi^{*}\phi$ commutes with each inner derivation of $\mfr{h}_{\mu}$. From this it follows that each one of its eigenspaces is $\ad[\mfr{h}_{\mu}]$-invariant, thus an ideal.

Even more, since the center and the derived ideal of $\mfr{h}_{\mu}$ are orthogonal and the transpose of each derivation is again a derivation the formula in \prop{2.6} shows that for each central $X$ we will have $M(\mu)X = 0$. In this case 

    \begin{equation*}
\phi^{*}\phi(x) = (k_{\mu,\phi}\id - M(\mu))(x) = k_{\mu,\phi}x.
    \end{equation*}

In the end we see that each eigenspace of $\phi^{*}\phi$ is an direct sum of the ideals $\mfr{h}_{i}$. Thus if we denote by $\phi_{i}$ the restriction of $\phi$ to $\mfr{h}_{\mu_{i}}$ we obtain \lastequation{0}. 

Now, let $(\mu_{i}:\phi_{i}) \in \mcal{L}_{n_{i}}$ be the pair corresponding to the Lie algebra $\mfr{h}_{i}$ and the restriction homomorphism $\phi_{i}: \mfr{h}_{i} \to \mfr{g}$. We must show that $(\mu_{i}:\phi_{i})$ is a critical pair. First, denote the vector $u_{\mu_{i},\phi_{i}}$  by $u_{i}$. Then it follows from the orthogonality of the ideals of $\mfr{h}_{\mu}$ that

    \begin{equation*}
\sum_{i,j=0}^{k} \langle u_{i}, u_{j} \rangle = \|u_{\mu,\phi}\|^{2} = 0.
    \end{equation*}

\noindent However if we take orthonormal bases $\{e_{k}\}$ and $\{e'_{l}\}$ for $\mfr{h}_{i}$ and $\mfr{h}_{j}$, respectively, it is also true that

    \begin{align*}
\langle u_{i}, u_{j} \rangle & = \sum_{k,l} \langle [\theta\phi_{i}(e_{k}), \phi_{i}(e_{k})], [\theta\phi_{j}(e'_{l}), \phi_{j}(e'_{l})] \rangle \\
& = -\sum_{k,l} \langle [\phi_{j}(e'_{l}),[\theta\phi_{i}(e_{k}), \phi_{i}(e_{k})]], \phi_{j}(e'_{l}) \rangle \\
& = -\sum_{k,l} \langle [[\phi_{j}(e'_{l}),\theta\phi_{i}(e_{k})], \phi_{i}(e_{k})], \phi_{j}(e'_{l}) \rangle \\
& = \sum_{k,l} \langle [[\phi_{j}(e'_{l}),\theta\phi_{i}(e_{k})], [\phi_{j}(e'_{l}), \theta\phi_{i}(e_{k})] \rangle \\
& \geqslant 0.
    \end{align*}

\noindent So we conclude that $u_{i} = 0$ for all $i$. 

Finally, it follows from \lastequation{0} that $\phi^{*}\phi|_{\mfr{h}_{i}} = \phi_{i}^{*}\phi_{i}$ and from the orthogonality relations for the ideals that $M(\mu) = \sum_{i=0}^{k} M(\mu_{i})$, thus

    \begin{equation*}
M(\mu_{i},\phi_{i}) = (M(\mu_{i}) - \phi_{i}^{*}\phi_{i},u_{i}) = ((M(\mu) - \phi^{*}\phi)|_{\mfr{h}_{i}},0) = -\frac{\|(\mu,\phi)\|^{2}}{n}(\id_{\mfr{h}_{i}},0).
    \end{equation*}

\noindent and we conclude that $(\mu_{i},\phi_{i})$ is a critical pair with minimal energy for all $i$.
    \end{proof}
    \end{proposition}

\subsubsection{Remark:} From the above result we can identify the cases for which $\mcal{S}_{\frac{1}{n}}$ will be empty by the merit of not existing a reductive pair $(\mu,\phi)$ for which the homomorphism is not injective in the center or its image is not toral. Indeed, this restriction will happen when both $n$ and the rank of $\mfr{g}$ are too small, which is the case exactly when $n=2$ and $\mfr{g}$ is either $\mfr{su}(2)$ or $\mfr{sl}(2)$. Latter we shall note that those are the only two cases for which $E(\mu:\phi) = \tfrac{1}{n}$ has no solution in $\mcal{L}_{n}(\mfr{g})$. \\

Also, from \prop{3.4}, we learn that one can reduce the study of critical pairs in the strata $\mcal{S}_{\frac{1}{n}}$ to the case of abelian and semi-simple pairs. With the former case being completely understood by \theo{3.2} we now work towards understanding the latter. 

%}
\subsection{Criticality of semi-simple pairs}

In the following we shall prove \theo{1.1}. As we observed in the introduction, this result will bring enough information about the semi-simple pairs and to finish the study of the global minimal for the energy functional. 

The variety $L_{n}^{0} \coloneqq L_{n}(0)$, where $0$ is the $0$-dimensional Lie algebra, parametrizes all Lie brackets in $\mbb{R}^{n}$. Note that $L_{n}^{0}$ naturally embeds inside $L_{n}(\mfr{g})$ as a closed subvariety by means of the equivariant map $\mu \mapsto (\mu,0)$; the same is true for the projectivization $\mcal{L}_{n}^{0} \hookrightarrow \mcal{L}_{n}(\mfr{g})$. The following lemma is a rewording of Theorem 4.3 in \cite{lauret2002momentmapvarietylie} more suitable for our needs.

    \begin{lemma}
Take $\mu \in \mcal{L}_{n}^{0}$ semi-simple. Then its $SC_{n}$-orbit is closed and $E(\mu:0) = \tfrac{1}{n}$ exactly when there is a Cartan involution $\theta'$ for $\mfr{h}_{\mu}$ with

    \begin{equation*}
\langle X,Y \rangle = - \frac{2n}{\|\mu\|^{2}}B_{\mu}(\theta' X,Y).
    \end{equation*}

\noindent In particular $(\mu,0)$ as a closed $SC_{n}(\mfr{g})$-orbit.

    \begin{proof}
Take $\theta''$ a Cartan involution for $\mfr{h}_{\mu}$ and then find an element $g \in GL(\mbb{R}^{n})$ for which the following relation holds:

    \begin{equation*}
\langle gX,gY \rangle = - \frac{2n}{\|\mu\|^{2}}B_{\mu}(\theta'' X,Y).
    \end{equation*}

Then if we put $\mu' = g \cdot \mu$ and $\theta' = g\theta'' g^{-1}$ we will obtain a semi-simple Lie algebra $\mfr{h}_{\mu'}$ with a Cartan involution $\theta'$ for which

    \begin{equation*}
\langle X,Y \rangle = -\frac{2n}{\|\mu\|^{2}}B_{\mu'}(\theta' X,Y).
    \end{equation*}

\noindent Now, from the formula for the moment map in \prop{2.6}, we may calculate $M(\mu')$:

    \begin{align*}
\langle M(\mu')X,Y \rangle & = \frac{1}{2}\sum_{i,j} \langle X,\mu'(e_{i},e_{j}) \rangle\langle \mu'(e_{i},e_{j}),Y \rangle - \sum_{i} \langle \mu'(X,e_{i}),\mu'(Y,e_{i}) \rangle \\
& = \frac{1}{2}\sum_{i} \langle \mu'(X,\theta' e_{i}), \mu'(Y,\theta' e_{i}) \rangle - \sum_{i} \langle \mu'(X,e_{i}),\mu'(Y,e_{i}) \rangle \\
& = -\frac{1}{2}\Tr \left( \ad^{\mu'*}_{x}\ad^{\mu'}_{Y}\right) \\
& = \frac{1}{2}B_{\mu'}(\theta' X,Y) \\
& = -\frac{\|\mu\|^{2}}{n} \langle X,Y \rangle
    \end{align*}

\noindent Thus $M(\mu') = -\tfrac{\|\mu'\|^{2}}{n}$, from which also follows that $\|\mu'\| = \|\mu\|$. This is enough to establish the result.
    \end{proof}
    \end{lemma}

Also, a very important fact about the semi-simple pairs is their Lie algebra of pair derivations. The situation is similar to the case of semi-simple Lie algebras, when most of the information is contained on $\mfr{h}_{\mu}$ itself.

    \begin{lemma}
Take $(\mu,\phi) \in \mcal{L}_{n}(\mfr{g})$ semi-simple. Then algebra of pair derivations splits as a product

    \begin{equation*}
\der(\mu,\phi) \cong \mfr{inn}(\mu,\phi) \oplus \mfr{z}_{\mfr{g}}(\ima\phi) \cong \der(\mu) \oplus \mfr{z}_{\mfr{g}}(\ima\phi).
    \end{equation*}

\noindent Consequently, we know its dimension

    \begin{equation*}
\dim\der(\mu,\phi) = \dim \mfr{z}_{\mfr{g}}(\ima\phi) + n.
    \end{equation*}

    \begin{proof}
First, observe that

    \begin{equation*}
\mfr{inn}(\mu,\phi) \cong \mfr{h}_{\mu} \cong \der(\mu),
    \end{equation*}

\noindent since by means of semi-simplicity, both $\ad^{\mu,\phi}$ and $\ad^{\mu}$ are isomorphisms into their images. Also, note that for every $v \in \mfr{z}_{\mfr{g}}(\ima\phi)$ we have $(0,v) \in \der(\mu,\phi)$, thus the centralizer of $\ima\phi$ is a subalgebra of $\der(\mu,\phi)$. In fact, we see that

    \begin{equation*}
\mfr{inn}(\mu,\phi) \cap \mfr{z}_{\mfr{g}}(\ima\phi) = \ker\ad^{\mu} = 0.
    \end{equation*}
    
Now, take $(\delta,v)$ in $\der(\mu,\phi)$. Then, by semi-simplicity, $\delta$ is an inner derivation of $\mfr{h}_{\mu}$, say we have $\delta = \ad^{\mu}_{X}$, so

    \begin{equation*}
(\delta,v) = (0,v-\phi(X)) + \ad^{\mu,\phi}_{X}.
    \end{equation*}

\noindent And for $Y \in \mfr{h}_{\mu}$ we have

    \begin{equation*}
[v-\phi(X),\phi(Y)] = \ad_{v}\phi(Y) - \phi\ad^{\mu}_{X}(Y) = 0.
    \end{equation*}

\noindent Hence $v-\phi(x) \in \mfr{z}_{\mfr{g}}(\ima\phi)$ and the lemma follows.
    \end{proof}    
    \end{lemma}

Then the culmination of our previous work will be the following theorem which establishes that every semi-simple pair is distinguished. From this result we shall prove the theorems in the introduction concerning the semi-simple pairs.

    \begin{proof}[Proof of \theo{1.1}]
    	
Let $(\mu,\phi) \in L_{n}$ be semi-simple, then the result of this theorem will follow from \lemm{3.1} if we prove that the $SC_{n}$-orbit of $(\mu,\phi)$ is closed. To do so, let $(\nu,\psi) \in L_{n}$ be such that

	\begin{equation}
SC_{n}\cdot(\nu,\psi) \subseteq \overline{SC_{n}\cdot(\mu,\phi)}
	\end{equation}

\noindent Thus if we prove equality in \lastequation{0} we will conclude that $SC_{n}$-orbit of $(\mu,\phi)$ is closed. To do so we claim that

    \begin{equation}
\dim\der(\mu,\phi) = \dim \der(\nu,\psi). 
	\end{equation}

\noindent from which will follow that the $SC_{n}$-orbits of $(\mu,\phi)$ and $(\nu,\psi)$ have the same dimension, and then \lastequation{1} would need to be an equality.

First, observe that \lastequation{1} implies that $SC_{n}\cdot\nu \subseteq \overline{SC_{n}\cdot\mu}$, thus we may assume, from \lemm{3.5}, that $\nu = \mu$. Then, by the Hilbert-Mumford criterion, there is an $\alpha \in \tilde{\mfr{p}}$ and $g \in GC_{n}$ such that

    \begin{equation}
\lim_{t\to\infty} e^{t\alpha} \cdot (\mu,\phi) = g\cdot(\mu,\psi).
    \end{equation}

\noindent However, from continuity, we also have $e^{t\alpha} \cdot \mu \to g\cdot\mu$ and since the $SC_{n}$-orbit of $\mu$ is closed we apply \lemm{2.10} to conclude that $\alpha = \left(\ad^{\mu}_{X},v\right)$ and $g = (h,g')$, where $X \in \mfr{h}_{\mu}$, $v \in \mfr{g}$, $h \in \aut(\mfr{h}_{\mu})$ and $g' \in \iner(\mfr{g})$. By exchanging $\psi$ for $g'\psi h^{-1}$ we may assume that $g = 1$.

Observe that $X$ must be a semi-simple element in $\mfr{h}_{\mu}$, since $\alpha = \left(\ad^{\mu}_{X},v\right)$ is in $\tilde{\mfr{p}}$. It follows that there exists a Cartan subalgebra $\mfr{t}$ of $\mfr{h}_{\mu}$ which contains $X$. From \lastequation{0} we obtain

    \begin{equation*}
\lim_{t\to\infty} e^{t(0,v)}\left( 0,\phi|_{\mfr{t}} \right) = \left( 0,\psi|_{\mfr{t}} \right).
    \end{equation*}

\noindent And then we may evoke \theo{3.2} and \lemm{2.10} to conclude that 

    \begin{equation}
\phi|_{\mfr{t}} = \psi|_{\mfr{t}}.
    \end{equation}

Now, a complexification of the objects involved shows that $\ad \circ \phi^{\mbb{C}}$ and $\ad \circ \psi^{\mbb{C}}$ are both equivalent representations $\mfr{h}_{\mu}^{\mbb{C}} \to \mfr{gl}\left( \mfr{g}^{\mbb{C}} \right)$. This follows imitatively from that fact that $\phi^{\mbb{C}}$ and $\psi^{\mbb{C}}$ are equal on the Cartan subalgebra $\mfr{t}^{\mbb{C}}$, by \lastequation{0}, so the weights and their multiplicities are equal for each representation. Alternatively, one can refer back to  Theorem 1.1 in \cite{Dynkin1957SemisimpleSO} as well. Then equation \lastequation{2} follows from \lemm{3.6} and because $\dim\mfr{z}_{\mfr{g}}(\ima\phi)$ equals the multiplicity of the trivial representation in $\ad \circ \phi^{\mbb{C}}: \mfr{h}_{\mu}^{\mbb{C}} \to \mfr{gl}\left( \mfr{g}^{\mbb{C}} \right)$ and, likewise, the same for $\dim\mfr{z}_{\mfr{g}}(\ima\psi)$.
    \end{proof}

As we have mentioned in the introduction, \theo{1.2} follows from \theo{1.1}. Indeed, from this result, and \lemm{3.1}, we know that the $SC_{n}$-orbit of any semi-simple pair is closed in $L_{n}$. In fact, we can conclude that the $GC_{n}$-orbit of a semi-simple pair is closed in the projectivization $\mcal{L}_{n}$. Then we obtain:

	\begin{proof}[Proof of \theo{1.2}]
Let $X$ be the collection of all semi-simple pairs in $\mcal{L}_{n}$. By Cartan's criterion on semi-simplicity, we know that $X$ is a semi-algebraic real variety which has finitely many components (Lojasiewicz‘s Theorem). On the other hand, we also know that $X$ is the disjoint union of $GC_{n}$-orbits, all of which are closed. It follows now that the number of $GC_{n}$-orbits in $X$ cannot be greater than the number of its connected components, thus it has to be finite. Finally, by taking complements of unions, it follows that each $GC_{n}$-orbit in $X$ is not only closed but also open in $X$. Since $X$ is open in $\mcal{L}_{n}$, again by Cartan's criterion, those orbits must be open in $\mcal{L}_{n}$ also.
	\end{proof}
 
Also, from the results above, we can finally know all the cases for which $E(\mu:\phi) = \frac{1}{n}$ has no solution. 

	\begin{proposition}
Aside from two exceptional cases, the minimal value of the energy functional $E: \mcal{L}_{n}(\mfr{g}) \to \mbb{R}$ is $\frac{1}{n}$. The only exceptions occur when $n=2$ and $\mfr{g}$ has rank 1, i.e., $\mfr{g}$ is either $\mfr{su}(2)$ or $\mfr{sl}(2)$. 

	\begin{proof}
The basic observation we need to make is that if $n\neq 2$ or if the rank of $\mfr{g}$ is not one then we can find a $\theta$-invariant Cartan subalgebra $\mfr{t} \leqslant \mfr{g}$ and a pair $(\mu,\phi) \in L_{n}(\mfr{g})$ with the following properties:

	\begin{itemize}
	
	\item The Lie algebra $\mfr{h}_{\mu}$ is reductive and the dimension of its center does not exceed the rank of $\mfr{g}$.
	
	\item The homomorphism $\phi: \mfr{h}_{\mu}: \mfr{h}_{\mu} \to \mfr{g}$ has as kernel the semi-simple part $\mfr{h}_{\mu}$ and its image is a subspace of $\mfr{t}$.
	
	\end{itemize}
	
\noindent Then it follows from the \prop{2.6}, \theo{3.2} and \lemm{3.5} that the energy functional attains the value $\tfrac{1}{n}$ in the orbit of $(\mu,\phi)$.

On the other hand, as in Remark 3.0.1, we know that $E(\mu:\phi) = \tfrac{1}{n}$ has no solution in the variety $\mcal{L}_{n}(\mfr{g})$ if $n=2$ and $\rank\mfr{g}=1$ since there will be no pair $(\mu,\phi)$ which satisfies the conclusions of \prop{3.4}.
	\end{proof}
	\end{proposition}   
 
\subsection{The Mostow Theorem}

We shall now finish this section and prove the Mostow Theorem (\theo{1.3}). Quite notably, this theorem follows quite naturally from a attempt to characterize the semi-simple critical pairs in the molds of \lemm{3.5}. We start with a simple lemma in linear algebra:

    \begin{lemma}
Let $V$ be an inner product space. Suppose that $T_{1}, \dots, T_{m}$ is a family of linear endomorphisms of $V$ such that

    \begin{equation*}
\sum_{i=1}^{m} [T_{i},T_{i}^{*}] = 0.
    \end{equation*}

\noindent Then to each $T_{i}$-invariant subspace $W$ the subspace $W^{\perp}$ is also $T_{i}$-invariant. 

    \begin{proof}
Let $\pi: V \to W$ be the orthogonal projection and let $\{e_{1}, \dots, e_{n}\}$ be an orthonormal basis for $V$ such that $\{e_{k+1}, \dots, e_{n}\}$ makes out a base for $W$. We calculate

    \begin{align*}
0 & = \sum_{i=1}^{m} \langle \pi, [T_{i},T_{i}^{*}] \rangle \\
& = \sum_{i=1}^{m} \langle [\pi,T_{i}], T_{i} \rangle \\
& = \sum_{i=1}^{m} \sum_{j=1}^{n} \langle [\pi,T_{i}](e_{j}), T_{i}(e_{j}) \rangle \\
& = \sum_{i=1}^{m} \sum_{j=1}^{k} \langle \pi T_{i}(e_{j}), T_{i}(e_{j}) \rangle
    \end{align*}

\noindent Since $\pi$ is a positive semi-definite self-adjoint operator of $V$ we see that

    \begin{equation*}
\pi T_{i}(e_{j}) = 0,
    \end{equation*}

\noindent for all $i$ and $j = 1, \dots, k$. This observation proves that $W^{\perp}$ is also $T_{i}$-invariant. 
    \end{proof}
    \end{lemma}

Then we now show how to construct for each simple subalgebra of $\mfr{g}$ a compatible Cartan involution. The main idea will be to take for a simple pair the element in its orbit which is an critical point of the energy and then take the Cartan involution provided by \lemm{3.5} and show it is compatible with $\theta$.

    \begin{proposition}

Take $(\mu:\phi) \in \mcal{L}_{n}$ simple, i.e., $\mfr{h}_{\mu}$ is simple. Then $E(\mu:\phi) = \tfrac{1}{n}$ exactly when there is a Cartan involution $\theta'$ for $\mfr{h}_{\mu}$ which satisfies

    \begin{equation}
\theta\phi = \phi\theta' \qquad \qquad \text{and} \qquad \langle X,Y \rangle = - \frac{2n}{\|\mu\|^{2}}B_{\mu}(\theta'(X),Y).
    \end{equation}

    \begin{proof}
First, assume that one has a Cartan involution $\theta'$ as in \lastequation{0}. Then immediately following \lemm{3.5} we have that

    \begin{equation*}
M(\mu) = -\frac{\|\mu\|^{2}}{n}\id
    \end{equation*}

\noindent Also, one easily sees that

    \begin{equation*}
u_{\mu,\phi} = 0.
    \end{equation*}

\noindent On the other hand, we have for every $X$:

    \begin{align*}
\phi^{*}\phi \ad^{\mu}_{X} & = \phi^{*} \ad_{\phi(X)}\phi \\
& = -\phi^{*} \ad_{\theta\phi(X)}^{*}\phi \\
& = -\left( \ad_{\phi\theta'(X)}\phi \right)^{*}\phi \\
& = -\left( \phi\ad_{\theta'(X)} \right)^{*}\phi \\
& = \ad_{X}\phi^{*}\phi \\
    \end{align*}

\noindent and then we see that $\phi^{*}\phi$ must be a multiple of the identity, by the simplicity of $\mfr{h}_{\mu}$. This is enough to conclude that $E(\mu,\phi) = \frac{1}{n}$.

Conversely, suppose that $E(\mu:\phi) = \tfrac{1}{n}$. We will then construct the required Cartan involution. In view of \prop{3.4} we may assume that $\phi$ is a isometry, up to rescaling of the pair. With those assumptions in place one look at the formula for the moment map of a pair reveals that $E(\mu:0) = \tfrac{1}{n}$ as well. Let $\theta'$ be the Cartan involution constructed in \lemm{3.5}, we shall argue that this is the one we are looking for. Indeed, we already have that $\theta'$ satisfies the right hand side of \lastequation{0}.

Observe that $\theta'$ has the property that $\ad_{X}^{\mu*} = -\ad_{\theta' X}^{\mu}$ for every $X$. Thus from \lemm{3.1} we see that $\left( \ad_{\theta'(X)}^{\mu},\theta\phi(X) \right)$ is a pair derivation which implies a certain compatibility between the two Cartan involutions:

    \begin{equation*}
\ad_{\theta\phi(X)}\phi = \phi\ad_{\theta'(X)}^{\mu} = \ad_{\phi\theta'(X)}\phi
    \end{equation*}

\noindent which shows that

    \begin{equation}
\ima(\theta\phi-\phi\theta') \leqslant \mfr{z}_{\mfr{g}}(\ima\phi).
    \end{equation}

Now, view $\mfr{g}$ as a $\mfr{h}_{\mu}$-module by means of the representation $\rho(X) = \ad_{\phi(X)}$. Of course, the image of $\phi$ will be a submodule. Since $u_{\mu,\phi} = 0$ we have

    \begin{equation*}
\sum_{i=1}^{n} [\rho(e_{i}),\rho(e_{i})^{*}] = 0
    \end{equation*}

\noindent for some orthonormal basis $\{e_{i}\}$ of $\mfr{h}_{\mu}$. Thus the hypothesis of \lemm{3.9} are in place and we are able to conclude that $\ima\phi^{\perp} = \ker\phi^{*}$ is a $\mfr{h}_{\mu}$-submodule complementary to $\ima\phi$. Furthermore, since, as $\mfr{h}_{\mu}$-modules, the centralizer $\mfr{z}_{\mfr{h}}(\ima\phi)$ and the image $\ima\phi$ are nonequivalent we must have 

    \begin{equation*}
\mfr{z}_{\mfr{g}}(\ima\phi) \leqslant \ker\phi^{*}.
    \end{equation*}

\noindent As a particular case of this inclusion, we take \lastequation{0} and the fact that $\phi$ is a isometry to see that

    \begin{equation}
\theta' = \phi^{*}\theta\phi
    \end{equation}

We finish by estimating, for each $X \in \mbb{R}^{n}$, the following quantity: 

    \begin{equation*}
\|(\phi\theta'-\theta\phi)(X)\|^{2} = \langle \phi\theta'(X),(\phi\theta'-\theta\phi)(X)\rangle - \langle \theta\phi(X),(\phi\theta'-\theta\phi)(X)\rangle.
    \end{equation*}

\noindent First we see that

    \begin{align*}
\langle \phi\theta'(X),(\phi\theta'-\theta\phi)(X)\rangle & = \langle \theta'(X),\phi^{*}(\phi\theta'-\theta\phi)(X)\rangle \\
& = \langle \theta'(X),(\theta'-\phi^{*}\theta\phi)(X)\rangle \\
& = 0
    \end{align*}

\noindent and then that

    \begin{align*}
\langle \theta\phi(X),(\phi\theta'-\theta\phi)(X)\rangle & = \langle \theta\phi(X),\phi\theta'(X)\rangle - \langle \theta\phi(X),\theta\phi(X)\rangle \\
& = \langle \phi^{*}\theta\phi(X),\theta'(X) \rangle - \langle X,X \rangle \\
& = \langle \theta'(X),\theta'(X) \rangle - \langle X,X \rangle \\
& = 0.
    \end{align*}

\noindent Therefore, $\|(\phi\theta'-\theta\phi)(X)\|^{2} = 0$ and \lastequation{2} holds.
    \end{proof}
    \end{proposition}

On the general case, a semi-simple pair $(\mu,\phi) \in L_{n}$ with energy $\tfrac{1}{n}$ will admit a Cartan involution $\theta': \mfr{h}_{\mu} \to \mfr{h}_{\mu}$ which is compatible with $\theta$ in the sense that

	\begin{equation*}
\theta\phi = \phi\theta'.
	\end{equation*}

\noindent This follows once we apply \theo{3.9} to each simple factor of $\mfr{h}_{\mu}$, noting they would be minimal points of the energy as well because of \prop{3.4}, and then extending each so obtained Cartan involution to one in $\mfr{h}_{\mu}$. With this fact, we can easily prove Mostow's Theorem.

	\begin{proof}[Proof of \theo{1.3}]
Let $(\mu,\varphi)$ be a semi-simple pair in $L_{n}(\mfr{g})$ and let $\sigma': \mfr{h}_{\mu} \to \mfr{h}_{\mu}$ be any given Cartan involution. By \theo{1.1}, we may find an $(g,f) \in GC_{n}$ for which the pair $(\nu,\psi) = (g,f)^{-1}(\mu,\phi)$ is critical and then by the cometary above, we can find an Cartan involution $\theta': \mfr{h}_{\nu} \to \mfr{h}_{\nu}$ for which 

	\begin{equation*}
\theta\psi = \psi\theta'.
	\end{equation*} 
	
\noindent On the other hand, note that $g\theta'g^{-1}$ is a Cartan involution for $\mfr{h}_{\mu}$, hence we may find an inner automorphism $g'$ for $\mfr{h}_{\mu}$ which satisfies

	\begin{equation*}
\sigma' = g'g\theta'(g'g)^{-1}.
	\end{equation*}
	
\noindent Additionally, relating to $g'$, there is another inner automorphism $f'$ for $\mfr{g}$ which satisfies $f'\phi = \phi g'$, so then $(g',f') \in \aut(\mu,\phi)$. Finally, we define $\sigma = f'f\theta(f'f)^{-1}$ and note that this is the Cartan involution for $\mfr{g}$ which satisfies

	\begin{equation*}
\sigma\phi = \phi\sigma'.
	\end{equation*}
	\end{proof}

\setcounter{section}{3}
\section{The general structure of a critical pair}
    % !TeX spellcheck = en_US

%{

We now study the general structure of critical pairs. Much of what can be said in this generality will come by studying the eigenvalues of the pair derivation $(D,u)_{\mu,\phi}$ and exploiting their interaction with both the algebraic and metric structure of the critical pairs.

Since we are mostly doing work with eigenvalues for different endomorphisms in distinct vector spaces, we will need to find appropriate bases to work with. This next lemma will take care of this.

    \begin{lemma}
Let $(\mu:\varphi) \in \mathcal{L}_{n}(\mfr{g})$ be critical. Then we can simultaneously diagonalize both $\phi^{*}\phi$ and $D$ with an orthonormal basis $\{e_{i} \suchthat 1 \leqslant i \leqslant n\}$ for $\mfr{h}_{\mu}$ which is ordered in such a way that

    \begin{equation*}
    \begin{cases}
\ker \phi = span\{e_{i} \suchthat 1 \leqslant i < p \text{ or }  q < i \leqslant r \} \\
\ker D = span\{e_{i} \suchthat 1 \leqslant i \leqslant q\}
    \end{cases}
    \end{equation*}

\noindent for $1 \leqslant p \leqslant q \leqslant r \leqslant n$. Also, the base elements may be chosen in such a way that there are $\varepsilon_{i} \in \{-1,1\}$ for which 

    \begin{equation*}
    \begin{cases}
\langle \mu(e_{i},X), Y \rangle = \varepsilon_{i} \langle X,\mu(e_{i},Y) \rangle \\
\theta\varphi(e_{i}) + \varepsilon_{i} \varphi(e_{i})  \in \mfr{z}_{\mfr{g}}[\phi[\ker D]]\\
    \end{cases}
    \end{equation*}

\noindent for every $X,Y \in \ker D$ and $1 \leqslant i \leqslant q$.

Furthermore, one can latter diagonalize $\ad_{u}$ with an orthonormal basis $\{\tilde{e}_{i} \suchthat q \leqslant i \leqslant s \}$ in such a way that 

    \begin{equation*}
\tilde{e}_{i} = \frac{\phi(e_{i})}{\|\phi(e_{i})\|}, \qquad p \leqslant i \leqslant  q \text{ and } r < i \leqslant n
    \end{equation*}

    \begin{proof}
Indeed, observe that

    \begin{align*}
D\varphi^{*}\varphi & = (\varphi^{*}\varphi D)^{*} \\
& = (\varphi^{*} \ad_{u} \varphi)^{*} \\
& = \varphi^{*} \ad_{u} \varphi \\
& = \varphi^{*}\varphi D \\
    \end{align*}

\noindent which shows that $D$ and $\phi^{*}\phi$ can be simultaneously diagonalized.

Now, since $D$ is a symmetric derivation the subspaces $\mfr{r} = \ker D$ and $\mfr{r}^{\perp} = \ima D$ will be a subalgebra and ideal of $\mfr{h}_{\mu}$, respectively. Also, recall from the proof of \prop{2.5} that for every $X \in \ker D$ we have

    \begin{equation*}
(-\ad^{\mu*}_{X},\theta\phi(X)) \in \der(\mu,\phi) \qquad \forall X \in \mfr{r}
    \end{equation*}

\noindent Hence \lemm{3.3} applies and we see that $\mfr{r}$ is a reductive subalgebra which splits orthogonally as $\mfr{r}' \oplus \mfr{a}$, where $\mfr{a}$ is the center of $\mfr{r}$, and those two subspaces are $\phi^{*}\phi$-invariant. Now, on every simple factor of $\mfr{r}'$ the map $\phi^{*}\phi$ is going to be a multiple of the identity, by the same lemma, so we can choose an orthonormal basis $\{e'_{i}\}$ and $\varepsilon_{i}' \in \{-1,1\}$ such that 

    \begin{equation}
\left.\ad_{e'_{i}}^{\mu*}\right|_{\mfr{r}} = 
\varepsilon'_{i}\left.\ad_{e'_{i}}^{\mu}\right|_{\mfr{r}}
    \end{equation}

\noindent Then note that if $X \in \mfr{r}$ then

    \begin{equation*}
[\theta\phi(e'_{i}) + \varepsilon_{i}\phi(e'_{i}),\phi(X)] = -\phi \ad_{e'_{i}}^{*}X + \varepsilon_{i}\phi\ad_{e'_{i}}X = \phi \left( \varepsilon_{i}\ad_{e'_{i}} - \ad_{e'_{i}}^{*} \right)(X) = 0
    \end{equation*}

\noindent Hence it follows that 

    \begin{equation}
\theta\phi(e'_{i}) + \varepsilon_{i}\phi(e'_{i}) \in \mfr{z}_{\mfr{g}}[\phi[\mfr{r}]]
    \end{equation}

We can repeat this construction for every simple factor of $\mfr{r}'$ and then re-index the base elements to obtain an orthonormal basis for $\mfr{r}'$ which satisfies both \lastequation{1} and \lastequation{0} and diagonalizes $\phi^{*}\phi$. Subsequently, we complete this base to one for $\mfr{h}_{\mu}$ by choosing an orthonormal basis of eigenvectors for and $D$ and $\phi^{*}\phi$ on $\mfr{a} \oplus \mfr{r}^{\perp}$ and obtain an orthonormal basis $\{e_{i}\}$ as in the statement of this lemma. This finishes the first part. 

\noindent For the second claim, note that if $X$ is vector with $D(X) = cX$ then we have

    \begin{align*}
\ad_{u}\phi(X) & = \phi(DX) = c\phi(X)
    \end{align*}

\noindent Hence following the construction of the base $\{e_{i}\}$ above it is easy to construct $\{\tilde{e}_{i}\}$.
    \end{proof}
    \end{lemma}

As in the case of Lie algebras, much of the structure of a critical pair is contained in its nilradical. Indeed, in this next proposition we show that the restriction to the nilradical of a critical pair is again critical. Then just after, we will work towards reverting this process by considering nilpotent pair and adjoint to it a reductive Lie algebra.

    \begin{theorem}
Let $(\mu,\varphi) \in L_{n}(\mfr{g})$ be a critical pair. Then there exists a (orthogonal) decomposition

    \begin{equation*}
\mfr{h}_{\mu} = \mfr{m}_{\mu} \oplus \mfr{a}_{\mu} \oplus \mfr{n}_{\mu}
    \end{equation*}

\noindent where $\mfr{m}_{\mu}$ is a Levi factor, $\mfr{a}_{\mu} \oplus \mfr{n}_{\mu}$ is the radical and $\mfr{n}_{\mu}$ is the nilradical of $\mfr{h}_{\mu}$. Furthermore, the restriction $(\nu,\psi)$ of $(\mu,\phi)$ to $\mfr{n}_{\mu}$ will be a critical pair with $u_{\nu,\psi} = u_{\mu,\phi}$ and $D_{\nu,\psi}$ as the restriction of $D_{\mu,\phi}$ to $\mfr{n}_{\mu}$.

    \begin{proof}    
Let $\mfr{n}_{\mu}$ be the image of $D$, which is an ideal, and, following \lemm{3.3}, note that the subspace $\ker D$ is a reductive subalgebra so we can define $\mfr{m}_{\mu}$ and $\mfr{a}_{\mu}$ to be its derived ideal and center, respectively. Observe that those spaces decompose $\mfr{h}_{\mu}$ orthogonally and that $\mfr{a}_{\mu} \oplus \mfr{n}_{\mu}$ and $\mfr{m}_{\mu}$ are the radical and a Levi factor, respectively. 

Also, note that $\mfr{a}_{\mu}$ does not contain nilpotent elements. For if $X \in \mfr{a}_{\mu}$ is nilpotent then the element $\ad^{\mu,\phi}_{X} = \left(\ad^{\mu}_{X},\phi(X)\right)$ of $\der(\mu,\phi)$ will also be nilpotent and because by assumption we have $D(X) = 0$ it follows that $\left(-\ad_{X}^{\mu*},\theta\phi(X)\right)$ will also be a nilpotent pair derivation. It is not hard to see that this is only possible if $\ad^{\mu,\phi}_{X}=0$. On the other hand, we proved in \prop{2.5} that if $X$ is a eigenvector of $D$ with $\ad^{\mu,\phi}_{X} = 0$ then $D(X) \neq 0$. This shows that $\mfr{n}_{\mu}$ is the nilradical of $\mfr{h}_{\mu}$.

Now, take an orthonormal basis $\{e_{i}\}$ for $\mfr{h_{\mu}}$ as in \lemm{4.1}. To prove that the restriction of $(\mu,\phi)$ to $\mfr{n}_{\mu}$, which we call $(\nu,\psi)$, is critical observe that 

	\begin{equation*}
M(\nu,\psi) + k_{\mu,\phi}(\id_{\mfr{n}},0) = (D|_{\mfr{n}_{\mu}},u) - (T|_{\mfr{n}_{\mu}},v)
	\end{equation*}

\noindent where

	\begin{equation*}
(T,v) = \sum_{i=1}^{q} \left([\ad^{\mu}_{e_{i}},\ad^{\mu*}_{e_{i}}], [\theta\phi(e_{i}),\phi(e_{i})] \right)
	\end{equation*}

\noindent is a pair derivation of $(\mu,\phi)$ (compare it with the formula in \prop{2.6}). Note also that $T(X) = 0$ for every $X \in \ker D$, so then it is the case that $\left(T|_{\mfr{n}_{\mu}},0 \right)$ is a pair derivation for $(\nu,\psi)$. 

Finally, we calculate

	\begin{align*}
\|(T,v)\|^{2} & = \| (T|_{\mfr{n}_{\mu}},v) \|^{2} \\
& = -\langle M(\nu,\psi) + k(\id_{\mfr{n}_{\mu}},0) - (D|_{\mfr{n}_{\mu}},u), (T|_{\mfr{n}_{\mu}},v) \rangle \\  
& = -\langle M(\nu,\phi), (T|_{\mfr{n}_{\mu}},0) \rangle - \langle M(\mu,\phi), (T,v) \rangle \\
& = 0.
	\end{align*}

\noindent Hence $(T,v)$ is zero and $(\nu,\psi)$ is a critical pair.
	\end{proof}
    \end{theorem}
   
Following this theorem, we can say the following about the eigenvalues of the pair derivation $(D,u)_{\mu,\phi}$:

    \begin{proposition}
Let $(\mu:\varphi) \in \mathcal{L}_{n}(\mfr{g})$ be a critical pair. Then both $\tfrac{1}{k}D_{\mu,\phi}$ and $\tfrac{1}{k}\ad_{u}$ have rational eigenvalues.

    \begin{proof}
Following \theo{4.2} we may assume without loss of generality that $D_{\mu,\phi}$ has no kernel. Then take an orthonormal basis for $\mfr{h}_{\mu}$ as in \lemm{4.1} and let $V$ be the subspace of $\End(\mbb{R}^{n}\oplus \mfr{g})$ generated by the pairs $(E_{i},\tilde{E}_{i})$ where $E_{i}: \mbb{R}^{n} \to \mbb{R}^{n}$ and $\tilde{E}_{i}: \mfr{g} \to \mfr{g}$ are defined by:

    \begin{equation*}
E_{i}(X) =
    \begin{cases}
\langle X, e_{i} \rangle e_{i}, & \qquad 1 \leqslant i \leqslant r \\
\tfrac{1}{2}\langle X, e_{i} \rangle e_{i}, & \qquad r < i \leqslant n \\
0, & \qquad n < i \leqslant s \\
    \end{cases}
\qquad \text{ and } \qquad
\tilde{E}_{i}(v) =
    \begin{cases}
0, & \qquad 1 \leqslant i \leqslant r \\
\tfrac{1}{2}\langle v, \tilde{e}_{i} \rangle \tilde{e}_{i}, & \qquad r < i \leqslant n \\
\langle v, \tilde{e}_{i} \rangle \tilde{e}_{i}, & \qquad n < i \leqslant s \\
    \end{cases}
    \end{equation*}

\noindent Observe that if $(A,B) \in V$ then $B\phi = \phi A$.

Additionally, consider the following subspace of $V$:

    \begin{equation*}
F = \text{Span}\left\{ (E_{i},\tilde{E}_{i}) + (E_{j},\tilde{E}_{j}) - (E_{k},\tilde{E}_{k}) \suchthat \langle\mu(e_{i},e_{j}),e_{k} \rangle \neq 0 \text{ or } \langle[\tilde{e}_{i},\tilde{e}_{j}],\tilde{e}_{k} \rangle \neq 0\right\}
    \end{equation*}

\noindent By observing the construction of $V$ and $F$, as well as the brackets $\mu$ and of $\mfr{g}$, it is a simply computation to see that

    \begin{equation}
F^{\perp} \cap V = \{(\delta,\ad_{v}) \suchthat (\delta,v) \in \der(\mu,\phi)\} \cap V
    \end{equation}

\noindent as a particular case we have $(D,\ad_{u}) \in F^{\perp} \cap V$.

The endomorphism $I \in V$ defined by

    \begin{equation*}
I = \sum_{i=1}^{q} \left\langle (\id,0), \left( E_{i},\tilde{E}_{i} \right) \right\rangle \left(E_{i},\tilde{E}_{i}\right)
    \end{equation*}

\noindent is the result of projecting $(\id,0)$ orthogonally onto $V$. Furthermore, if $P: V \to F$ is the orthogonal projection, then: 

    \begin{equation}
\text{Both }I \text{ and } P(I) \text{ have only rational eigenvalues.}
    \end{equation}

\noindent This is so because the endomorphisms $\left( E_{i},\tilde{E}_{i}\right)$ all have only rational eigenvalues and the vector spaces $V$ and $F$ are generated by them.  

Finally, consider the following element of $V$:

    \begin{equation*}
\tilde{M} = (D,\ad_{u}) -k_{\mu,\phi}I. 
    \end{equation*}

\noindent If any pair $(\delta,\ad_{v}) \in F^{\perp} \cap V$ is given then we see that $(\delta,v) \in \der(\mu,\varphi)$ and so we apply \lemm{2.3} to show that

    \begin{align*}
\langle \tilde{M},(\delta,\ad_{v}) \rangle & = \langle (D,\ad_{u}) -k_{\mu,\phi}I,(\delta,\ad_{v}) \rangle \\
& = \langle (D,\ad_{u}) -k_{\mu,\phi}(\id,0),(\delta,\ad_{v}) \rangle \\
& = \langle M(\mu) - \varphi^{*}\varphi, \delta \rangle + \langle \ad_{u},\ad_{v} \rangle \\
& = \langle M(\mu) - \varphi^{*}\varphi, \delta \rangle + \langle u,v \rangle \\
& = \langle M(\mu,\varphi),(\delta,v) \rangle \\
& = 0.
    \end{align*}

\noindent Thus by \lastequation{1} we have $\tilde{M} \in F$. Hence, we have $\tilde{M} = -k_{\mu,\phi}P(I)$ which implies that

    \begin{equation*}
\frac{1}{k_{\mu,\phi}}(D,\ad_{u}) = I + \frac{1}{k_{\mu,\phi}}\tilde{M} = I - P(I). 
    \end{equation*}

\noindent and the result follows from \lastequation{0}.
    \end{proof}
    \end{proposition}

It follows from this previous proposition that for a critical pair $(\mu,\phi)$ there would be some common constant $c>0$ for which all the eigenvalues of $cD_{\mu,\phi}$ and $c\ad_{u_{\mu,\phi}}$ are integers. Thus if we take $\sum_{i \in \mbb{Z}} \mfr{g}^{i}$ to be the eigenspace decomposition for $c\ad_{u_{\mu,\phi}}$ we will obtain a gradation of $\mfr{g}$. Analogously, we will obtain an $\mbb{Z}_{\geqslant0}$-gradation $\sum_{i \geqslant 0}\mfr{h}_{\mu}^{i}$ for $\mfr{h}_{\mu}$ if we take the eigenspaces for $cD_{\mu,\phi}$. Finally, we note that the homomorphism $\phi$ takes $\mfr{h}_{\mu}^{i}$ into $\mfr{g}^{i}$, this follows directly from the fact that $c(D,u)_{\mu,\phi}$ is a pair derivation. This proves \theo{1.4}.   

%}

	\subsection{Extension of critical pairs} 

Now we shall work towards a converse for \theo{4.2}. In other words, we shall develop a method which takes a critical nilpotent pair and attaches to it a reductive pair in such a manner to remain critical. In fact, we shall make this construction for any pair which is itself critical and has as codomain a appropriate Lie algebra.

In the following we will need to consider the situation of pairs in $L_{n}(\mfr{r})$, where $\mfr{r}$ is an $\tilde{\theta}$-invariant subalgebra of $\mfr{gc}_{n}(\mfr{g})$ equipped with the inner product produced from $\tilde{\beta}$ and $\tilde{\theta}$ (see \eq{2.1}). In this case, without change, most of the definitions and results we encountered so far will remain true. The only exception will be the result of \theo{4.2}, which is no longer true, where we used that the adjoint representation $\ad:\mfr{g} \to \der(\mfr{g})$ is an isometry. But most importantly, the definition and considerations we made in \S2 about critical pairs remains true.

	\begin{proposition}
Let $(\mu,\varphi) \in L_{n}(\mfr{g})$ be a critical pair whose nilradical has dimension $m$. Then the pair $(\nu,\psi)$, where $\nu$ is the restriction of $\mu$ to the subalgebra $\mfr{l}_{\mu} \coloneqq \mfr{m}_{\mu} \oplus \mfr{a}_{\mu}$ and $\psi$ is the homomorphism $\mfr{l}_{\mu} \to \mfr{gc}_{m}(\mfr{g})$ given by

	\begin{equation*}
\psi(X) = \left( \left.\ad^{\mu}_{X}\right|_{\mfr{n}_{\mu}},\phi(X) \right)
	\end{equation*} 

\noindent is a critical pair in $L_{m}(\mfr{gc}_{n}(\mfr{g}))$ with minimal energy and such that $k_{\nu,\psi} = k_{\mu,\phi}$.

	\begin{proof}
From \prop{2.6} we may compare the moment maps for $\mu$ and $\nu$. To do so take an orthonormal basis $\{e_{i}\}$ for $\mfr{n}_{\mu}$ and note that:

	\begin{equation}
\langle M(\nu)X,Y \rangle = \langle M(\mu)X,Y \rangle + \sum_{i=1}^{m} \langle \mu(X,e_{i}),\mu(Y,e_{i}) \rangle
	\end{equation}

\noindent for $X,Y \in \mfr{l}_{\mu}$. However, the inner product on $\mfr{gc}_{m}(\mfr{g})$ is such that

	\begin{equation*}
\langle \psi(X),\psi(Y) \rangle = \sum_{i=1}^{m} \langle \mu(X,e_{i}),\mu(Y,e_{i}) \rangle + \langle \phi(X),\phi(Y) \rangle.
	\end{equation*}

\noindent Hence \lastequation{0} could be rewritten as
	
	\begin{equation*}
\langle M(\nu)X,Y \rangle - \langle \psi(X),\psi(Y) \rangle = \langle M(\mu)X,Y \rangle - \langle \phi(X),\phi(Y) \rangle
	\end{equation*}

\noindent from which we conclude, because $\mfr{l}_{\mu}$ is the kernel of $D_{\mu,\phi}$, that

	\begin{equation*}
M(\nu) - \psi^{*}\psi + k_{\mu,\phi}\id = 0
	\end{equation*}

\noindent On the other hand, it should be clear from \lemm{4.1}, that

	\begin{equation*}
u_{\nu,\psi} = 0.
	\end{equation*}

\noindent Therefore, putting everything together, we obtain

	\begin{equation*}
M(\nu,\psi) + k_{\mu,\phi}(\id,0) = 0.
	\end{equation*}

\noindent which proves the assertion that $(\nu,\psi)$ is a global minimum for the energy.
	\end{proof}
	\end{proposition}

Hence if we want to revert \theo{4.2} we should look at pairs in $L_{m}(\mfr{gc}_{n}(\mfr{g}))$ whose image are pair derivations for $(\mu,\phi) \in L_{n}(\mfr{g})$. Indeed, we look at the subalgebra

	\begin{equation*}
r_{\mu,\phi} \coloneqq \der(\mu,\phi) \cap \tilde{\theta}[\der(\mu,\phi)]
	\end{equation*}

\noindent which is the largest $\tilde{\theta}$-invariant subalgebra of $\der(\mu,\phi)$. Naturally, $r_{\mu,\phi}$ is always reductive and, in fact, if we take the projection $\pi: \mfr{r}_{\mu,\phi} \to \mfr{g}$ the corresponding pair will be a critical pair in some $L_{n'}(\mfr{g})$. Furthermore, if the pair $(\mu,\phi)$ then Corollary 9.2 in \cite{bohm2017realgeometricinvarianttheory} show that each element in $\mfr{r}_{\mu,\phi}$ centralizes $m(\mu,\phi)$ and, consequently, we should have
	
	\begin{equation}
\qquad [(\delta,v),(D,u)_{\mu,\phi}] = 0, \qquad \forall (\delta,v) \in \mfr{r}_{\mu,\phi}
	\end{equation}

We now look at the pairs in $L_{m}(\mfr{r}_{\mu,\phi})$ and use them to extend $(\mu,\phi)$. Indeed, to each pair $(\nu,\psi) \in L_{m}(\mfr{r}_{\mu,\phi})$ we will produce a new pair in $L_{m+n}(\mfr{g})$ by defining

	\begin{equation*}
\tilde{\nu}((X,A),(Y,B)) = (\nu(X,Y),\mu(A,B) + (\delta \circ \psi)(X)(B) -  (\delta \circ \psi)(Y)(A)),
	\end{equation*}
	
\noindent where $\delta: \der(\mu,\phi) \to \der(\mu)$ is the projection, and 

	\begin{equation*}
\tilde{\psi}(X,A) = (\pi \circ \psi)(X) + \phi(A)
	\end{equation*}

\noindent Then the result $(\nu,\psi) \ltimes (\mu,\phi) \coloneqq (\tilde{\nu},\tilde{\phi})$ will be a pair in $L_{m+n}(\mfr{g})$ which has the structure of the orthogonal semi-direct product $\mfr{h}_{\nu} \ltimes \mfr{h}_{\mu}$.

	\begin{theorem}
Let $(\mu,\phi) \in L_{n}(\mfr{g})$ be a nilpotent critical pair. Then to each other critical pair $(\nu,\psi) \in L_{m}(\mfr{r}_{\mu,\phi})$ the semi-direct product $(\tilde{\nu},\tilde{\psi}) \coloneqq (\nu,\psi) \ltimes (\mu,\phi)$ will be a critical pair in $L_{n}(\mfr{g})$, provided we make a normalization to have $k_{\nu,\psi} = k_{\mu,\phi}$.

	\begin{proof}

Let $k$ be the common value of $k_{\mu,\phi}$ and $k_{\nu,\psi}$. Then note that

	\begin{equation}
u_{\tilde{\nu},\tilde{\psi}} = \pi(u_{\nu,\psi}) + u_{\mu,\phi}.
	\end{equation}
	
\noindent And we need to find a similar formula for $D_{\tilde{\nu},\tilde{\psi}}$.

Since $\mfr{h}_{\nu}$ and $\mfr{h}_{\mu}$ are orthogonal subspaces of $\mfr{h}_{\tilde{\nu}}$ we may find for the latter an orthonormal basis $\{e_{i}, \dots, e_{m+n}\}$ for which $\{e_{1}, \dots, e_{m}\}$ and $\{e_{m+1}, \dots, e_{m+m}\}$ span $\mfr{h}_{\nu}$ and $\mfr{h}_{\mu}$, respectively. Then we may find that the moment map for $\tilde{\nu}$ decomposes as

	\begin{align*}
\langle M(\tilde{\nu})(X,A),(Y,B) \rangle & = \langle M(\nu)X,Y \rangle + \langle M(\mu)A,B \rangle \\
& + \sum_{i=1}^{m} \langle [\delta(\psi(e_{i})),\delta(\psi(e_{i}))^{*}]A,B \rangle - \sum_{i=m+1}^{n} \langle \delta(\psi(X))e_{i},\delta(\psi(Y))e_{i} \rangle
	\end{align*} 

But then we note that

	\begin{equation*}
\sum_{i=1}^{m} [\delta(\psi(e_{i})),\delta(\psi(e_{i}))^{*}] = \delta(u_{\nu,\psi})
	\end{equation*}

\noindent and 

	\begin{equation*}
\sum_{i=m+1}^{n} \langle \delta(\psi(X))e_{i},\delta(\psi(Y))e_{i} \rangle = \langle \psi(X),\psi(Y) \rangle - \langle \tilde{\psi}(X),\tilde{\psi}(Y) \rangle.
	\end{equation*}
	
\noindent Hence

	\begin{align*}
\langle M(\tilde{\nu})(X,A),(Y,B) \rangle & = \langle (\psi^{*}\psi+k\id+D_{\nu,\psi})X,Y \rangle + \langle (\phi^{*}\phi + k\id+\delta(u_{\nu,\psi})+D_{\mu,\phi})A,B \rangle \\
& + \langle \tilde{\psi}(X),\tilde{\psi}(Y) \rangle - \langle \psi(X),\psi(Y) \rangle \\
& = \langle (k\id+D_{\nu,\psi})X,Y \rangle + \langle (k\id+\delta(u_{\nu,\psi})+D_{\mu,\phi})A,B \rangle \\
& + \langle \tilde{\psi}(X),\tilde{\psi}(Y) \rangle + \langle \phi(A),\phi(B) \rangle \\
& = \langle (k\id+D_{\nu,\psi})X,Y \rangle + \langle (k\id+\delta(u_{\nu,\psi})+D_{\mu,\phi})A,B \rangle \\
& + \langle \tilde{\psi}(X,A),\tilde{\psi}(Y,B) \rangle
	\end{align*}
	
\noindent And then from definition we conclude that

	\begin{equation*}
D_{\tilde{\nu},\tilde{\psi}} = D_{\nu,\psi} \oplus \left(\delta(u_{\nu,\psi})+D_{\mu,\phi}\right)
	\end{equation*}

To finish we just need to verify that $(D,u)_{\tilde{\nu},\tilde{\psi}}$ is a pair derivation. To do so we need to verify two conditions:

	\begin{itemize}
	
	\item The endomorphism $D_{\tilde{\nu},\tilde{\psi}}$ is a derivation of $\mfr{h}_{\tilde{\nu}}$. To verify this condition note it is enough to calculate $D_{\tilde{\nu},\tilde{\psi}}\tilde{\nu}(X,A)$ for $X \in \mfr{h}_{\nu}$ and $A \in \mfr{h}_{\mu}$. In this case we have
	
		\begin{align*}
	D_{\tilde{\nu},\tilde{\psi}}\tilde{\nu}(X,A) & = \left( \delta(u_{\nu,\psi}) + D_{\mu,\phi} \right)\delta(\psi(X))A \\
	& = \left( \delta[u_{\nu,\psi},\psi(X)] + \delta(\psi(X))\delta(u_{\nu,\psi}) + D_{\mu,\phi}\delta(\psi(X)) \right) A
		\end{align*}
		
	\noindent But then from \lastequation{0} we have $D_{\mu,\phi}\delta(\psi(X)) = \delta(\psi(X))D_{\mu,\phi}$, thus
	
		\begin{align*}
	D_{\tilde{\nu},\tilde{\psi}}\tilde{\nu}(X,A) & = \left( \delta[u_{\nu,\psi},\psi(X)] + \delta(\psi(X))\delta(u_{\nu,\psi}) + \delta(\psi(X))D_{\mu,\phi} \right) A \\
	& = \left( \delta\psi(D_{\nu,\psi}X) + \delta(\psi(X))\left( \delta(u_{\nu,\psi}) + D_{\mu,\phi} \right) \right)A \\
	& = \tilde{\nu}(D_{\tilde{\nu},\tilde{\psi}}X,A) + \tilde{\nu}(X,D_{\tilde{\nu},\tilde{\psi}}A)
		\end{align*}¨
		
	\item The homomorphism $\tilde{\phi}$ interchanges $D_{\tilde{\nu},\tilde{\psi}}$ and $[u_{\tilde{\nu},\tilde{\psi}},-]$. For this we apply \lastequation{0} again:

		\begin{align*}
	\tilde{\psi}D_{\tilde{\nu},\tilde{\psi}} (X,A) & = \tilde{\psi}(D_{\nu,\psi} X, (\delta(u_{\nu,\psi})+D_{\mu,\phi})A) \\
	& = \pi(\psi D_{\nu,\psi} X) + \phi(\delta(u_{\nu,\psi})A) + \phi(D_{\mu,\phi}A)  \\
	& = \pi[u_{\nu,\psi},\psi(X)] + [\pi(u_{\nu,\psi}),\phi(A)] + [u_{\mu,\phi},\phi(A)] \\
	& = \pi[u_{\nu,\psi}+(D,u)_{\mu,\phi},\psi(X)] + [\pi(u_{\nu,\psi}),\phi(A)] + [u_{\mu,\phi},\phi(A)] \\
	& = [u_{\tilde{\nu},\tilde{\psi}},\tilde{\psi}(X,A)],
		\end{align*}

	\end{itemize}

\noindent Thus $(D,u)_{\tilde{\nu},\tilde{\psi}}$ is a pair derivation for $(\tilde{\nu},\tilde{\psi})$ and the theorem follows. 
	\end{proof} 
	\end{theorem}

%%% --------- %%%}
%%% Appendix %%%{
\appendix

%    \printbibliography[title={Bibliography}]

 \bibliographystyle{alpha}
 \bibliography{Bibliography}
    
%%% --------- %%%}

\end{document}